\documentclass[a4paper,oneside,11pt]{article}
\usepackage{amsmath}
\usepackage{amssymb}
\usepackage{indentfirst}
\usepackage{abstract}
\usepackage{titlesec}
\usepackage{graphicx}
\usepackage{fancyhdr}
\usepackage{latexsym}
\usepackage{mathrsfs}
\usepackage{booktabs}
\usepackage{mathrsfs}
\usepackage{mathtools}
\usepackage{amsthm}
\usepackage{wasysym}
\usepackage{cite}
\usepackage{color}
\usepackage{todonotes}
\usepackage{geometry}

\numberwithin{equation}{section}
\newtheorem{definition}{Definition}[section]
\newtheorem{theorem}{Theorem}[section]
\newtheorem{proposition}{Proposition}[section]
\newtheorem{lemma}{Lemma}[section]
\newtheorem{remark}{Remark}[section]
\newenvironment{solution}{\begin{proof}[\bf Proof of Proposition \ref{p1}]}{\end{proof}}
\newcommand{\curl}{{\rm curl} }

\renewcommand{\div}{ {\rm div }  }
\title{Strong solutions to the initial-boundary-value problem of compressible MHD equations with degenerate viscosities and far field vacuum  in 3D exterior domains
}
\author{Jiaxu Li$^{a}$, Boqiang L\"u$^{b}$, Bing Yuan$^{b}$, \thanks{email: jiaxvlee@gmail.com (J.X. Li), bingyuan@email.ncu.edu.cn (B. Yuan), lvbq86@163.com (B.Q.   L\"u) } \\
{\normalsize a. School of Mathematical Sciences,}\\
{\normalsize Shenzhen University, Shenzhen 518061, P. R. China;}\\
{\normalsize b. School of Mathematics and Computer Sciences}\\ {\normalsize \& Institute of Mathematics and Interdisciplinary Sciences,}\\ {\normalsize  Nanchang University, Nanchang 330031, P. R. China;}
}
\date{}

\begin{document}
\maketitle
\begin{abstract}
This paper concerns the initial-boundary-value problem (IBVP) of the compressible Magnetohydrodynamic (MHD) equations in 3D exterior domains with Navier-slip boundary conditions for the velocity and perfect conducting conditions for the magnetic field. For the case that the density approaches far-field vacuum initially and the viscosities are power functions of the density ($\rho^\delta$ with $0<\delta<1$), the local existence and uniqueness of strong solutions to the IBVP is established for regular large initial data. 
In particular, in contrast to the local theory of compressible Navier-Stokes equation Li-L\"u-Yuan\cite{lly26ns}, we show that the magnetic field  maintain  the initial quality of decaying faster rate than density throughout the time evolution, which   reveals the role of the magnetic field in handling singularities arising from density-dependent viscosities. 
\end{abstract}



\section{Introduction}

The motion of a viscous compressible fluid under the effect of an electromagnetic field is governed by the compressible MHD equations:
\begin{equation}\label{1}
\left\{ \begin{array}{l}
  \rho_t+\mathrm{div}(\rho u)=0,\\
  (\rho u)_t+\mathrm{div}(\rho u\otimes u)+\nabla P=\mathrm{div}\mathbb{T}+H\cdot\nabla H-\nabla H\cdot H, \\
  H_{t}- \eta \Delta H=\mbox{curl}(u\times H),\quad\\
  \mbox{div}H=0,
  \end{array} \right.
\end{equation}
Here, $t\geq0$ is time, $x=(x_1,x_2,x_3)\in\Omega\subset\mathbb{R}^3$ is the spatial coordinate, and the unknown functions $\rho=\rho(x,t)$, $u=$ $(u^1,u^2,u^3)(x,t)$, $P(\rho)=a\rho^\gamma (a>0, \gamma>1)$, and $H=(H^1, H^2, H^3)(x,t)$ denote the density, velocity, pressure, and the magnetic field, respectively.
$\mathbb{T}$ denotes the viscous stress tensor in the form
\begin{equation}
\mathbb{T}=2\mu(\rho)\mathcal{D}(u)+\lambda(\rho)\mathrm{div}u\mathbb{I}_3,
\end{equation}
where $\mathcal{D}(u)=\frac12\big[\nabla u+(\nabla u)^T\big]$ is the deformation tensor, $\mathbb{I}_3$ is the $3\times 3$ identity matrix. The shear viscosity $\mu(\rho)$ and the bulk one $\lambda(\rho)+\frac{2}{3}\mu(\rho)$ satisfy the following hypothesis:
\begin{equation}\label{vis-c}
\mu(\rho)=\mu\rho^\delta,\ \lambda(\rho)=\lambda\rho^\delta,
\end{equation}
for some constant $\delta\geq 0$, $\mu$ and $\lambda$ are both constants satisfying
\begin{equation}\label{5}
\mu>0,\ 2\mu+3\lambda\geq0.
\end{equation}
The constant $\eta>0$ is the magnetic diffusivity describing a magnetic diffusion coefficient of the magnetic field.

Let $\Omega=\mathbb{R}^3-\bar{D}$ be the exterior of a simply connected bounded domain $D$ in $\mathbb{R}^3$, we look for a strong solution $(\rho,u,H)$ with finite energy to the IBVP of system \eqref{1} in exterior domain $\Omega$ with:
\begin{itemize}
  \item Initial data:
  \begin{equation}
(\rho,u, H)|_{t=0}=(\rho_0\geq0,u_0,H_0),\ x\in\Omega.
\end{equation}
  \item Navier-slip boundary conditions for the velocity $u$:
  \begin{equation}
    \label{ch1}
	u\cdot n=0,\quad \mbox{curl}u\times n=-A(x)u,\quad \mbox{on}\quad\partial\Omega,
\end{equation}
where $n=(n^1,n^2,n^{3})$ is the unit normal vector to the boundary $\partial\Omega$ pointing outside $\Omega$,
$A$ is a $3\times3$ symmetric matrix defined on $\partial\Omega$.
  \item Perfect conducting conditions for magnetic field $H$:
  \begin{equation}\label{chuy}
H\cdot n=0,\quad \mbox{curl}H\times n=0,\quad \mbox{on}\quad \partial\Omega.
\end{equation}
  \item Far field behavior:
\begin{align}\label{7}
(\rho,u,H)\rightarrow(0,0,0),\quad \mbox{as}\quad |x|\rightarrow\infty.
\end{align}
\end{itemize}

\begin{remark}
In general, Navier-type slip condition derived in \cite{Navier1827} can be stated as follows:
\begin{align}\label{ch2}
u\cdot n=0,\,\,\,\ \big(2\mathcal{D}(u) n+\vartheta u\big)_{tan}=0, \,\,\,\text{on} \,\,\,\partial\Omega,
\end{align}
where $\theta$ is a scalar friction function which measures the tendency of the fluid to slip on the boundary, and the symbol $v_{tan}$ represents the projection of tangent plane of the vector $v$ on $\partial\Omega$. In fact, the boundary condition \eqref{ch2} is equivalent to \eqref{ch1} in the sense of the distribution (see \cite{{Cai2023}} for details).
\end{remark}

The MHD system \eqref{1} models the dynamics of electrically conducting fluids, such as plasmas, and incorporates the effects of magnetic fields on fluid motion. If this motion occurs in the absence of a magnetic field, that is $H =0$, the MHD system \eqref{1} becomes the well-known Navier–Stokes equations for compressible fluids, which have been studied extensively in  \cite{Nash1962,Serrin1959,Tani19774,3,4,5,13,Straskraba1993,matsumura1980initial,Lions1998-2,Feireisl2001,jiang1,jiang2,HuangLiXin2012,Li201977,hhpz,Cai2023,Cai2112,9,Vasseur2016,bresch2004some,7,6,MR4340224,zhuxin2021,Lili2025,lly26ns,liu2025ns} and the references therein. For the constant viscous flows ($\delta=0$ in \eqref{vis-c}), the local well-posedness of smooth solutions is established in \cite{Nash1962,Serrin1959,Tani19774} (in the absence of vacuum)  and \cite{3,4,5,13,Straskraba1993} (with intial vacuum).
The global classical solutions were first obtained by Matsumura-Nishida \cite{matsumura1980initial} for initial data close to a non-vacuum equilibrium in $H^3$. In the presence of initial vacuum, the global existence of solutions becomes much more challenging due to the degeneration of momentum equation. The major breakthrough is due to Lions \cite{Lions1998-2}, where he obtained the global existence of finite-energy weak solutions when the adiabatic exponent $\gamma\ge9/5$ for the 3D case, which was further refined to the critical case $\gamma > 3/2$ by Feireisl \cite{Feireisl2001}. For spherically symmetric or axisymmetric initial data, Jiang-Zhang \cite{jiang1,jiang2} proved the global existence of weak solutions with vacuum for any $\gamma>1$. Under some suitable smallness assumptions on the initial data, the global smooth solutions with initial vacuum has been investigated in \cite{HuangLiXin2012,Li201977,hhpz,Cai2023,Cai2112} and the references therein. For  the Cauchy problem, 
Huang-Li-Xin \cite{HuangLiXin2012} and Li-Xin \cite{Li201977}  established the  3D and 2D global  classical solution  with  small total energy but possibly large oscillations, Hong-Hou-Peng-Zhu \cite{hhpz}
prove the 3D global  classical solutions under the assumptions that  $\gamma-1$ is small. Considering the IBVP with slip boundary conditions on velocity, 
Cai-Li \cite{Cai2023} and Cai-Li-L\"u \cite{Cai2112} obtained global classical solutions with small total energy in 3D bounded domain and exterior domain, respectively.
 
For density-dependent viscosities ($\delta>0$ in \eqref{vis-c}),  the multi-dimensional compressible Navier-Stokes system received extensive attentions in recent years, in which the well-posedness of solutions become more challenging due to the degenerate viscosity at vacuum. Recently, Li-Xin \cite{9} and Vasseur-Yu \cite{Vasseur2016} independently investigated the global weak solutions for compressible Navier-Stokes systems adhere to the Bresch-Desjardins relation \cite{bresch2004some}.  
When $\delta \ge 1$ in \eqref{vis-c}, Li-Pan-Zhu \cite{7,6} obtained the local existence of classical solution to the 2D/3D Cauchy problem with far field vacuum. Then, the 3D local classical solutions ($\delta > 1$) has been extended to be a global one by  Xin-Zhu \cite{MR4340224} for a class of smooth initial data that are of small density. In the case of $\delta\in(0,1)$  in \eqref{vis-c}, using different methods, Xin-Zhu \cite{zhuxin2021} and Li-Li \cite{Lili2025} established the local classical solution to the 3D  Cauchy problem with far field vacuum. Very recently, using different methods and strategies, Li-L\"u-Yuan\cite{lly26ns} and Liu-Zhong \cite{liu2025ns} obtained the local strong solution with far field vacuum to  the IBVP in  3D exterior domain with Navier-slip boundary conditions. It's should be mentioned here that, for the problem  \cite{zhuxin2021,Lili2025,lly26ns,liu2025ns} with $\delta\in (0,1)$, one has to handle the  the strong  singularity of $\nabla\rho^{\delta-1}$ and $\rho^{\delta-1}$   when the density $\rho\rightarrow 0$, which
is essentially different from the case $\delta\ge1$ considered in \cite{7,6,MR4340224}.


Then, back to the MHD system \eqref{1}, whose mathematical structure and physical mechanism are more complicated due to the fluid dynamic motion and the magnetic field interact strongly on each other. Here, we briefly review some results on the well-posedness of solutions with vacuum for the multi-dimensional compressible MHD equation, which are more closely related to the topic of this paper. When the viscosity coefficients $\mu$ and $\lambda$ are constant,  the local existence of strong solutions with initial vacuum is shown in \cite{Fan2009,Lv2015}. 
In the Lions' framework \cite{Lions1998-2}, Hu-Wang \cite{huwang2010} obtained the global existence of finite energy renormalized weak solutions with suitably large adiabatic exponent $\gamma$ in 3D bounded domains. Liu-Yu-Zhang \cite{lyz2013jde} established the global weak solutions with small total energy to the 3D Cauchy problem when the far-field density is away from vacuum. For the Cauchy problem with far filed vacuum, Li-Xu-Zhang \cite{lxz2013siam} and L\"u-Shi-Xu \cite{lsxiumj}  proved the global 3D and 2D classical solutions with initial data which are of small energy but possibly large oscillations respectively, Hong-Hou-Peng-Zhu \cite{hgyi} established the 3D global classical solutions under some smallness assumptions on $\gamma-1$ and $\eta^{-1}$.
Recently,  Chen-Huang-Peng-Shi \cite{Chen20241} and Chen-Huang-Shi \cite{Chen2024} established the global existence of classical solutions with small energy to the IBVP in 3D bounded domain and exterior one.

Compared with the constant viscosity coefficients, it is more difficult to investigate the well-posedness of solutions to the MHD system \eqref{1} with density-dependent viscosities and vacuum.  
When the shear viscosity coefficient $\mu$ is a positive constant and the bulk one $\lambda(\rho)=\rho^\beta~(\beta>4/3)$, Wang-Xu \cite{wangxuejde} established the global strong solutions to the 2D Cauchy problem with vacuum as the far field density. Very recently, for the density-dependent viscosities \eqref{vis-c} with $\delta=1$, Liu-Luo-Zhong \cite{liu2025} obtained the local  classical solution  with far field vacuum to the 3D Cauchy problem and IVBP in 3D exterior domains, which partially extends the results in Li-Pan-Zhu \cite{7} for compressible Navier-Stokes
equations to those of the compressible MHD equations.  However, for the case $\delta\in (0,1)$, 
the approaches used in \cite{liu2025}  do  not apply   due to  the strong  singularity   of $\rho^{\delta-1}$ and some new  difficulties from the magnetic field, and  there are few works about the system \eqref{1} with vacuum for the IBVP in 3D exterior domains even for local solution. 
Therefore, for the  MHD system \eqref{1}  with density-dependent viscosities $\delta\in (0,1)$, we aim to establish the local strong solution  with far field vacuum to the IBVP in 3D  exterior domains.

The strong solutions to the IBVP \eqref{1}-\eqref{7}  considered in this paper is defined   as follows:
\begin{definition}
Let $T>0$ be a finite constant. A solution $(\rho, u, H)$ to the system \eqref{1}-\eqref{7} is called a strong solution if all the derivatives involved in \eqref{1} for $(\rho, u, H)$ are regular distributions, and \eqref{1} holds almost everywhere in $[0, T]\times \Omega$.
\end{definition}

Before stating the main results, we first explain the notations and conventions used throughout this paper. For a positive integer $k$ and $p\geq1$, we denote the standard Lebesgue and Sobolev spaces as follows:
$$
\begin{gathered}
\|f\|_{L^p}=\|f\|_{L^p(\Omega)},\ \|f\|_{W^{k,p}}=\|f\|_{W^{k,p}(\Omega)},\ \|f\|_{H^k}=\|f\|_{W^{k,2}},\ \|f\|_{L^{p}}=\|f\|_{W^{0,p}},\\
D^{k,p}=\{f\in L^1_{loc}(\Omega):|f|_{D^{k,p}}=\|\nabla^kf\|_{L^p}<\infty\},\ D^k=D^{k,2},\\
D_0^{1}(\Omega)=\{f\in L^{6}(\Omega):\|\nabla f\|_{L^2}<\infty\},\ \|f\|_{X\cap Y}=\|f\|_X+\|f\|_Y.
\end{gathered}
$$
Let $B_R=\{x\in\mathbb{R}^3||x|<R\}$, we define  $$\Omega_R\triangleq \Omega\cap B_R$$
where  $R>2R_0+1$ with $R_0$ is  chosen to be sufficiently large such that $\bar D\subset B_{R_0}$. In particular, we denote $\Omega_0\triangleq \Omega\cap B_{2R_0}$. Furthermore, we define the new variables as follows:
\begin{align}\label{xq1q0}
\psi\triangleq\nabla\log\rho,\quad J\triangleq\rho^{-\frac{1+\delta}{2}}H. 
\end{align}
Here, the new function $J$ is the  density-weighted average magnetic field and plays a crucial role in our paper.

The main result of this paper is the following Theorem \ref{T1} concerning the local existence of strong solutions.

\begin{theorem}\label{T1}
For parameters $(\gamma,\delta)$ satisfy
\begin{equation}
\gamma>1,\ 0<\delta<1. \label{115}
\end{equation}
If the initial data $(\rho_0, u_0, H_0)$ satisfies
\begin{equation}
\begin{cases}\label{1117}
\rho_0^\frac{1-\delta}{2}u_0\in L^2,\ u_0\in D^1\cap D^2,\ H_0\in H^{2},\ \div{H_0}=0,\\
J_0\triangleq H_0/\rho_0^\frac{1+\delta}{2}\in H^{1},\ \rho_0^{\gamma-\frac{1+\delta}{2}}\in {D_0^1\cap D^2},\ \nabla\rho_0^\frac{\delta-1}{2}\in D_0^{1}, 
\end{cases}
\end{equation}
and the compatibility condition:
\begin{align}\label{118}
\mathcal{L}(u_0)=\rho_0^\frac{1-\delta}{2}g
\end{align}
for some $g\in L^2$. Then there exist a positive time $T_0>0$ such that the  IBVP \eqref{1}-\eqref{7} has a unique strong solution $(\rho, u, H)$ on $[0,T_0]\times \Omega$ satisfying 
\begin{equation}\label{e116}
\left\{ \begin{array}{l}
\rho^{\gamma-\frac{1+\delta}{2}}\in L^\infty([0,T_0];{W^{1,6}\cap D^1\cap D^2}),\\ \nabla\rho^\frac{\delta-1}{2}\in L^\infty([0,T_0];D_0^{1}),\quad \rho^\frac{\delta-1}{2}\in L^\infty([0,T_0];L^6(\Omega_0)),\\\
u\in L^\infty([0,T_0];D_0^1\cap D^2)\cap L^2([0,T_0];D^3),\\
H\in L^\infty([0,T_0];H^2)\cap L^2([0,T_0];D^3),\ J\in L^\infty([0,T_0];H^1)\cap L^2([0,T_0];D^2),\\
u_t\in L^2([0,T_0];D_0^1),\ H_t\in L^2([0,T_0];H^1),\ J_t\in L^2([0,T_0];L^2),\\
\rho^\frac{1-\delta}{2}u_t,\ H_t \in L^\infty([0,T_0];L^2),\\
\rho^\frac{\delta-1}{2}\mathcal{L}u,\ \rho^\frac{\delta-1}{2}\nabla\mathrm{div}u\in L^2([0,T_0];H^1).
\end{array} \right.
\end{equation}
\end{theorem}


\begin{remark}
The conditions \eqref{115}-\eqref{1117} in Theorem \ref{T1} identify a class of admissible
initial data that makes the problem \eqref{1}-\eqref{7} solvable, which satisfy by, for example,
\begin{equation}\nonumber
\rho_0(x)=\frac{1}{1+|x|^{2\alpha}},\ u_0(x)\in C_0^2(\Omega),\
H_0(x)=\left(\frac{1}{1+|x|^{2\delta_1}},\frac{1}{1+|x|^{2\delta_2}},\frac{1}{1+|x|^{2\delta_3}}\right),
\end{equation}
where $\frac{1}{2(2\gamma-1-\delta)}<\alpha<\frac{1}{2(1-\delta)}$ and $\delta_i>\frac{1+\delta}{4}\alpha+\frac34 \ (i=1,2,3)$. In particular, the range of $\delta$ is independent of $\gamma$, which means Theorem \ref{T1} applies to all $\gamma>1$ and $\delta\in(0,1)$. And, the magnetic field $H$ may have compact support.
Moreover, if $2\gamma\leq\delta+\frac{4}{3}$, then $\rho_0\in L^1(\Omega)$, and the density $\rho$ constructed in Theorem \ref{T1} may have finite total mass.
\end{remark}

\begin{remark} It deduces from $J_0\triangleq H_0/\rho_0^\frac{1+\delta}{2}\in H^{1}$ that  $H_0$ decays  faster than $\rho_0^\frac{1+\delta}{2}$ with respect to $x$. Indeed, Theorem \ref{T1} gives $J\in L^\infty([0,T_0];H^1)$, which illustrates that   the magnetic field $H$ will maintained the quality of decaying faster rate than $\rho^\frac{1+\delta}{2}$  throughout the time evolution. This new observation  plays a crucial role in our  analysis for controlling the singular terms associated with the magnetic field.  
\end{remark}

\begin{remark}\label{re4}
In fact, based on the choice of $\rho_0$ (see \eqref{1117}) and Lemma \ref{087}, we can obtain that $\rho_0$ has a positive lower bound on $\Omega_0$, and therefore it has
\begin{equation}\label{260114}
    \rho_0^\frac{\delta-1}{2}\in L^6(\Omega_0).
\end{equation}Theorem \ref{T1} gives a new priori estimate
$\rho^\frac{\delta-1}{2}\in L^\infty([0,T_0];L^6(\Omega_0))$, which combining with $\nabla\rho^\frac{\delta-1}{2}\in L^\infty([0,T_0];D_0^1)$ gives that $\rho$ has a positive lower bound near the boundary $\partial\Omega$. Roughly speaking, the density  retains a positive lower bound on $\Omega_0$ as initial one.
\end{remark}
\begin{remark}
For the time continuity, similar to \cite[Remark 1.2]{Lili2025}, we can also deduce from \eqref{e116} and the classical Sobolev embedding results that
\begin{gather}
    \rho^{\gamma-\frac{1+\delta}2}\in C([0, T_0];D_0^1\cap D^2),\ \nabla\rho^{\frac{\delta-1}{2}}\in C([0, T_0];D_0^1),\label{lxxp}
    \\
    u\in C([0,T_0];D_0^1\cap D^2),\
    \rho^\frac{1-\delta}2u,\ \rho^\frac{1-\delta}2u_t\in C([0,T_0];L^2),\nonumber\\
    H\in C([0,T_0];H^2),\
    H_t\in C([0,T_0];L^2),\ J\in C([0,T_0];H^1).\nonumber
\end{gather}
Furthermore, \eqref{lxxp} together with Lemma \ref{087} implies that a vacuum can only occur at infinity.
\end{remark}
\begin{remark}
 When there is no electromagnetic field effect, that is  $H=0$, \eqref{1} turns to be the compressible Navier-Stokes equations, and Theorem \ref{T1} is similar to the results of \cite{lly26ns}. Roughly speaking, we generalize the results of \cite{lly26ns} to the compressible MHD system.
\end{remark}


\begin{remark} It's should be noted here that one can   use the 
similar arguments in  Theorem \ref{T1} under some slight modifications to establish the local  strong solution to the Cauchy problem of system \eqref{1}, in which the density-weighted average magnetic field is defined by $\tilde J=H/\rho^\frac{1+\delta}{4}$ and  some additional boundary estimates ( such as \eqref{1qaqa4},  \eqref{ze313}) are not needed.  
\end{remark}

\begin{remark}
For the zero magnetic dissipation problem of system \eqref{1} ($\eta=0$), $H$  satisfies a hyperbolic equation and does not need to be given boundary conditions. That is, there is no difficulty comes from the boundary. we can also obtain a local strong solution to the system \eqref{1} with $\eta=0$.
\end{remark}

We now provide our analysis and commentary on the key aspects of this paper. 
First, motivated by \cite{Lili2025} (see also \cite{lly26ns}), we multiply \eqref{1}$_2$ by $\rho^{-\delta}$  and then reformulate \eqref{1}$_2$ as
\begin{equation}
\begin{split}\label{xq1q}
\rho^{1-\delta}(u_t+u\cdot\nabla u)-\mathcal{L}u=&-\rho^\frac{1-\delta}2\Big(\frac{2a\delta}{2\gamma-1-\delta}\nabla\rho^{\gamma-\frac{1+\delta}{2}}
-\frac{2\delta}{\delta-1}\nabla\rho^\frac{\delta-1}{2}\cdot \mathcal{S}(u)\\
&-\rho^{-\frac{1+\delta}{2}}H\cdot\nabla H+\rho^{-\frac{1+\delta}{2}}\nabla H\cdot H\Big),
\end{split}
\end{equation}
where
\begin{equation}\label{q1q1}
\mathcal{L}u=\mu\Delta u+(\mu+\lambda)\nabla \mathrm{div}u,\ \mathcal{S}(u)=2\mu\mathcal{D}(u)+\lambda \mathrm{div}u\mathbb{I}_3.
\end{equation}
In order to handle the singular terms 
\begin{equation}\label{jxlbq}
\rho^{-\frac{1+\delta}{2}}H\cdot\nabla H,~~~~~\rho^{-\frac{1+\delta}{2}}\nabla H\cdot H
\end{equation} 
on the right hand of \eqref{xq1q}, using the new variables $\psi$ and the density-weighted average magnetic field $J$ defined in \eqref{xq1q0}, it deduce from $\eqref{1}_1$ and $\eqref{1}_3$  that $J$    satisfies the following standard parabolic equation
\begin{equation}\label{xq1q2}
    J_t+u\cdot\nabla J-\eta\Delta J=J\cdot\nabla u+\frac{\delta-1}{2}J\mbox{div} u+\frac{1+\delta}{2}\eta J\mbox{div}\psi+\frac{(1+\delta)^2}{4}\eta J\psi^2+(1+\delta)\eta\psi\cdot\nabla J,
\end{equation}
and the following boundary conditions
\begin{equation}\label{bdy-J}
J\cdot n=0,\quad \curl J\times n=-\frac{1+\delta}{2}(\psi\cdot n)J,\quad \mbox{on}\quad \partial\Omega.
\end{equation}
In fact, by direct calculations, we obtain
\begin{align*}
\curl J=\rho^{-\frac{1+\delta}{2}}\curl H+\nabla\rho^{-\frac{1+\delta}{2}}\times H=\rho^{-\frac{1+\delta}{2}}\curl H-\frac{1+\delta}{2}\psi\times J,
\end{align*}
which combined with \eqref{chuy} yields
\begin{align*}
\curl J\times n=-\frac{1+\delta}{2}\psi\times J\times n=-\frac{1+\delta}{2}(\psi\cdot n)J.
\end{align*}


Next, with the reformulate equations \eqref{xq1q} and \eqref{xq1q2} at hand, the system \eqref{1} could be transformed into the following equations on $(\rho,u,H,J)$:
\begin{equation}
\begin{cases}\label{q1q}
  \rho_t+\mathrm{div}(\rho u)=0,\\
\rho^{1-\delta}(u_t+u\cdot\nabla u)-\mathcal{L}u=\\
\quad-\rho^\frac{1-\delta}2\Big(\frac{2a\gamma}{2\gamma-1-\delta}\nabla\rho^{\gamma-\frac{1+\delta}{2}}
-\frac{2\delta}{\delta-1}\nabla\rho^\frac{\delta-1}{2}\cdot \mathcal{S}(u)-J\cdot\nabla H+\nabla H\cdot J\Big),\\
H_{t}- \eta \Delta H=\curl(u\times H),\quad \div H=0,\\
J_t+u\cdot\nabla J-\eta\Delta J=\\
\quad J\cdot\nabla u+\frac{\delta-1}{2}J\mbox{div} u+\frac{1+\delta}{2}\eta J\mbox{div}\psi+\frac{(1+\delta)^2}{4}\eta J\psi^2+(1+\delta)\eta\psi\cdot\nabla J.
\end{cases}
\end{equation}
Therefore, we aim to search the strong solutions to \eqref{q1q} and thus establish  the strong solutions  of  system \eqref{1}. 

To obtain a strong solution of the system  \eqref{q1q},  we will establish  some uniform a priori estimates on the smooth solutions with positive initial density to the  approximate problem \eqref{9}  in an annular domain $\Omega_R$, see Lemma \ref{l25}. Inspired by Li-L\"u-Yuan\cite{lly26ns}, we also employ the cut-off function technique in dealing with the $\div$-$\curl$ estimates and elliptic estimates on $\Omega_R$. This along with the fact that $\rho$ has a positive lower bound near the boundary, plays a crucial role in our subsequent analysis concerning the multi-connected annular domain $\Omega_R$  and the boundary terms. However, there are new difficulties degenerated from the $H$ and $J$. 

As mentioned in \cite{Lili2025,lly26ns}, one of the key estimates is  to obtain the $L^\infty(0,T_0;D_0^{1}(\Omega_R))$-norm of $\nabla\rho^\frac{\delta-1}{2}$, which is dominated by $\rho^\frac{\delta-1}{2}\mathcal{L}u$ and $\rho^\frac{\delta-1}{2}\curl\curl u$, see  \eqref{zsa1}--\eqref{zsa2}. Based on the results in \cite{lly26ns} concerning the compressible Navier-Stokes equations, it suffices to estimate the terms involving  $H$ and $J$, where  the estimation of $H$ could be derived directly from the standard parabolic equation \eqref{q1q}$_3$  without singular terms. Consequently, our main effort will be devoted to obtaining   the estimations of $J$  involving a second-order  term of $\rho$, that is $\div \psi=\div (\nabla \log\rho)$, see \eqref{q1q}$_4$. Among these,  the primary challenge is to address the boundary terms associated with $J$. More precisely, since the boundary condition \eqref{bdy-J} of $J$ depends on $\psi$, one has following  derivation to the boundary term of $J$:
\begin{align}\label{pmy1}
\int_{\partial\Omega_R}(\mbox{curl}J\times n)\cdot J_tdS=-\frac{1+\delta}{2}\int_{\partial\Omega}(\psi \cdot n)J \cdot J_tdS.
\end{align}
Obviously, this boundary term can not be handled directly by the trace theorem due to the lack of  integrability of $\nabla J_t$.   Fortunately, one can transform  $J_t$ into $H_t$ as follows
$$
J_t=H_t\rho^{-\frac{1+\delta}{2}}-Hu\cdot\nabla\rho^{-\frac{1+\delta}{2}}+\frac{1+\delta}{2}H\rho^{-\frac{1+\delta}{2}}\div u.
$$
Combined this with the integrability of $\nabla H_t$ (see \eqref{pop11} and \eqref{hqaz}) and the fact that $\rho$ has a positive lower bound near the boundary $\partial\Omega$ (see \eqref{rho_lb}), one can apply the  trace theorem and thus derive  the desired estimate for the boundary term \eqref{pmy1}, see \eqref{ze313}.

Finally, in order to obtain the uniqueness of the solution, 
we need to estimate $\bar J$ $(\bar J=J_1-J_2)$, where one has to deal with   the  second-order derivative of $\rho$. 
To circumvent this  difficulty,   we replace  $\rho_i^\frac{1-\delta}{2}J_i$ by $g_iH_i$ with $g_i=\rho_i^{-\delta}$ (see \eqref{2.96}),  which thus  transforms the   estimate of $\bar J$  into the ones of   $\bar H=H_1-H_2$ and $\bar g=g_1-g_2$. 
However,   we lack any useful information concerning  $g_i=\rho_i^{-\delta}$ ($\rho$ with   negative power). 
To overcome this difficulty, we note that the initial value of  $\bar g$ is 0, that is $\bar g|_{t=0}=0$, which  together with  \eqref{lop1} and standard arguments thus  completes the proof of  the  uniqueness.

The rest of the paper is organized as follows: In Section 2, we collect some elementary facts and inequalities that will be needed in later analysis. Section 3 is devoted to the a priori estimates that are needed to obtain the local existence and uniqueness of strong solutions. Then the main result, Theorem 1.1 is proved in Section 4.

\section{Preliminaries}
In this section, some useful known facts and inequalities are listed. The first is the well-known Gagliardo-Nirenberg inequality \cite{10}.
\begin{lemma}\label{G-N}
Assume that $U\subset \mathbb{R}^3$ is a bounded Lipschitz domain, for $p\in[2,6]$, $q\in(1,\infty)$, and $r\in(3,\infty)$, there exist generic constants $C$,~$C_1$,~$C_2$ which depend on $p, q, r,$ and the Lipschitz character of $U$, such that for any
$f\in H^1(U),\ and\ g\in L^q(U)\cap D^{1,r}(U),$
one has
\begin{equation}
\|f\|_{L^p(U)}\leq C\|f\|^{(6-p)/2p}_{L^2(U)}\|\nabla f\|^{(3p-6)/2p}_{L^2(U)}+C_1\|f\|_{L^2(U)},
\end{equation}
\begin{equation}
\|g\|_{C(\bar{U})}\leq C\|g\|^{q(r-3)/(3r+q(r-3))}_{L^q(U)}\|\nabla g\|^{3r/(3r+q(r-3))}_{L^r(U)}+C_2\|g\|_{L^q(U)}.
\end{equation}
Moreover, if $f\cdot n|_{\partial U} = 0$, we can choose $C_1 = 0$. Similarly, the constant $C_2 = 0$ provided $g\cdot n|_{\partial U} = 0$.
\end{lemma}

The following lemma, whose proof is established in \cite{Aramaki2014Lp, Cai2112,lly26ns}, plays a crucial role in our analysis. 
\begin{lemma}\label{l22}
For any $u\in W^{1,p}(\Omega_R)$ with  $p\in(1,\infty)$, there exists some positive constant $C$ depending only on $p$ and $R_0$ such that
\begin{itemize}
    \item When $u\cdot n=0$ on $\partial\Omega_R$, one has
\begin{align}\label{uwkq}
\|\nabla u\|_{L^{p}(\Omega_R)}\leq C\bigr(\|\div u\|_{L^{p}(\Omega_R)}+\|\curl u\|_{L^{p}(\Omega_R)}+\|u\|_{L^p(\Omega_{0})}\bigr),
\end{align}
\item when $u\cdot n \neq 0$ on $\partial\Omega_R$, one has
\begin{align}\label{uwkq1}
\|\nabla u\|_{L^{p}(\Omega_R)}\leq C\bigr(\|\div u\|_{L^{p}(\Omega_R)}+\|\curl u\|_{L^{p}(\Omega_R)}+\|u\|_{L^{p}(\Omega_0)}+\|u\cdot n\|_{W^{1-1/p,p}(\partial\Omega_R)}\bigr).
\end{align}
\end{itemize}
\end{lemma}

Considering the Lam\'e's system
\begin{equation}\label{2.3}
\begin{cases}
		\mathcal{L}u=f,~~~~~&x\in \Omega_R,\\
		u\cdot n = 0,~\mathrm{curl}u\times n=-Au,~~&x\in\partial \Omega,\\
		u\cdot n=0,~\mathrm{curl}u\times n=0, & x\in\partial B_R,
	\end{cases}
\end{equation}
thanks to \cite{Cai2023,Lili2025}, we have the following conclusions on the solutions to \eqref{2.3}.
\begin{lemma}\label{elliptic_estimate}
Let $u$ be a smooth solution of the Lam\'e system \eqref{2.3}, for any $p\in(1,\infty)$ and $k\geq 0$, there exists a positive constant $C$ depending only on $\lambda$, $\mu$, $A$, $R_0$, $p$, and $k$ such that
\begin{equation}\label{key}
\|\nabla^{k+2}u\|_{L^p(\Omega_R)}\leq C\|f\|_{W^{k,p}(\Omega_R)}+C\|\nabla u\|_{L^2(\Omega_R)}.
\end{equation}
\end{lemma}

Next, we will give the following regularity estimate  on the solutions to  Poisson equation with the specific boundary condition:
\begin{equation}\label{uyt}
\left\{
\begin{array}{lll}
\Delta J=F,\ \ &x\in \Omega_R,\\
J\cdot n=0,\quad \mbox{curl}J\times n=-\frac{1+\delta}{4}(\psi\cdot n)J,\ \ &x\in\partial \Omega,\\
J=0,\ \ &x\in\partial B_R.
\end{array}
\right.
\end{equation}
\begin{lemma}\label{l24}
Let $J$ be a smooth solution to the  problem \eqref{uyt}, then there exists a positive constant $C$ depending only on $R_0$ such that
\begin{equation}\label{lkop}
	\|\nabla^2J\|_{L^2(\Omega_R)}\leq C\|F\|_{L^2(\Omega_R)}+\big(1+\|\psi\|^2_{D_0^{1}(\Omega_R)}\big)\|J\|_{H^1(\Omega_R)}.
\end{equation}
\end{lemma}
\begin{proof}
Using the truncation technique (see \cite{Cai2112}), we only need to consider the following equation in bounded domains $\Omega_{0}$:
\begin{equation}\label{2.5}
\left\{
\begin{array}{lll}
-\Delta J=F,~~~~~~~~~~~~~~~~~~~~~~~~~~~~~~~~~~~~~ &x\in \Omega_{0},\\
J\cdot n=0,~ \mbox{curl}J\times n=-\frac{1+\delta}{2}(\psi\cdot n)J,~~ &x\in\partial \Omega,\\
J=0,~ \nabla J=0,~~~~~~~~~~~~~~~~~~~~~~~~~~~~~~ &x\in\partial B_{2R_0}.
\end{array}
\right.
\end{equation}
First, it follows from \cite[Proposition 2.6]{Aramaki2014Lp} that
\begin{equation}\label{a2.161}
\begin{split}
&\|\nabla^2 J\|_{L^2(\Omega_{0})}\\
&\leq C\|\nabla\div J\|_{L^2(\Omega_{0})}+C\|\nabla\curl J\|_{L^2(\Omega_{0})}+C\|J\|_{H^1(\Omega_{0})}\\
&\leq C\|\nabla\div J\|_{L^2(\Omega_{0})}+C\|\nabla\times\curl J\|_{L^2(\Omega_{0})}+C\|\curl J\times n\|_{H^{1/2}(\partial\Omega_{0})}+C\|J\|_{H^1(\Omega_{0})},
\end{split}
\end{equation}
where the boundary term can be controlled as
\begin{equation}\label{2.161}
\begin{split}
 \|\mbox{curl}J\times n\|_{H^{1/2}(\partial\Omega_{0})} 
&\leq C\|\psi\cdot n J\|_{H^{1}(\Omega_{0})}\\
&\leq C\|\psi\|_{L^{6}(\Omega_{0})} \|J\|_{L^{3}(\Omega_{0})}+C\|\psi\|_{D_0^1(\Omega_{0})}(\|J\|_{L^\infty(\Omega_{0})}+\|\nabla J\|_{L^3(\Omega_{0})})\\
&\leq \epsilon\|\nabla^2 J\|_{L^2(\Omega_{0})}+\big(1+\|\psi\|^2_{D_0^{1}(\Omega_{0})}\big)\|J\|_{H^1(\Omega_{0})}.
\end{split}
\end{equation}

Next, one can rewrite the equation \eqref{2.5}$_1$ as
\begin{align}\label{equation-J}
\nabla\times\mbox{curl}J-\nabla\mbox{div}J=F.
\end{align}
Multiplying  \eqref{equation-J} by $\nabla\mbox{div}J$ and integrating the resultant equality over $\Omega_0$, it holds
\begin{align}\label{J-1}
\|\nabla\mbox{div}J\|_{L^2(\Omega_{0})}^2=\int_{\Omega_{0}}\nabla\times\mbox{curl}J\cdot\nabla\mbox{div}Jdx-\int_{\Omega_{0}}F\cdot\nabla\mbox{div}Jdx.
\end{align}
Direct calculations together with integration by parts and the boundary conditions $\eqref{2.5}_2$-$\eqref{2.5}_3$ lead to
\begin{align}\label{J-3}
\int_{\Omega_{0}}\nabla\times\mbox{curl}J\cdot\nabla\mbox{div}Jdx&=-\int_{\partial\Omega_0}(\mbox{curl}J\times n)\cdot\nabla\mbox{div}JdS\nonumber\\
&=\frac{1+\delta}{2}\int_{\partial\Omega_0}(\psi\times J\times n)\cdot\nabla\mbox{div}JdS\nonumber\\
&=\frac{1+\delta}{2}\int_{\Omega_0}\mbox{div}(\psi\times J\times \nabla\mbox{div}J)dx\nonumber\\
&\leq C\|\nabla\times(\psi\times J)\|_{L^2(\Omega_{0})}\|\nabla\mbox{div}J\|_{L^2(\Omega_{0})}\nonumber\\
&\leq \epsilon\|\nabla\mbox{div}J\|_{L^2(\Omega_{0})}^2+\big(1+\|\psi\|^4_{D_0^{1}(\Omega_{0})}\big)\|J\|_{H^1(\Omega_{0})}^2.
\end{align}

Therefore, putting  \eqref{J-3} into \eqref{J-1} and choosing $\epsilon$ small enough, we conclude that
\begin{align}\label{div}
\|\nabla\mbox{div}J\|_{L^2(\Omega_{0})}\leq C\|F\|_{H^1(\Omega_{0})}+C\big(1+\|\psi\|^2_{D_0^{1}(\Omega_{0})}\big)\|J\|_{H^1(\Omega_{0})},
\end{align}
and thus
\begin{align}\label{curl}
\|\nabla\times\mbox{curl}J\|_{L^2(\Omega_{0})}= \|\nabla\mbox{div}J+F\|_{L^2(\Omega_{0})}\leq C\|F\|_{H^1(\Omega_{0})}+C\big(1+\|\psi\|^2_{D_0^{1}(\Omega_{0})}\big)\|J\|_{H^1(\Omega_{0})}.
\end{align} 
This combined with \eqref{a2.161}, \eqref{2.161}, and  \eqref{div} immediately implies the desired estimate \eqref{lkop} and completes the proof of Lemma \ref{l24}.
\end{proof}

Finally, we need the following lemma to establish the positive lower bound of the density near the boundary, whose proof is established in \cite{lly26ns}.

\begin{lemma}\label{087}
Assume that $U\subset\mathbb{R}^3$ is a bounded smooth domain satisfying the interior sphere condition, if $\rho\in C(\overline U)$ and 
$\nabla\rho^{-a}\in L^p(U)$ with  $a>0,\ p>2$, then $\rho$ has a positive lower bound  on $\overline U$.
\end{lemma}

\section{A priori estimate}
In this section, for $1\leq p\leq\infty$ and positive integer $k$, we denote
\begin{equation*}
	\int f \mathrm{dx}=\int_{\Omega_R} f \mathrm{dx},\ L^p=L^p(\Omega_R),\ W^{k,p} = W^{k,p}(\Omega_R),\ D^{k}=D^{k,2}(\Omega_R),\ H^k=H^k(\Omega_R).
\end{equation*}

We start with the following local well-posedness result which can be proved in the same way as \cite{3,4,5}, where the initial density is strictly away from vacuum.
\begin{lemma}\label{l25}
Assume that the initial data $(\rho_0,u_0,H_0,J_0)$ satisfies
\begin{equation}\label{2.16}
\begin{cases}
\rho_0,\ u_0,\ H_0\in H^3(\Omega_R),\ J_0\in H^2(\Omega_R),\ \div H_0=0,\ \inf_{x\in \Omega_R}\rho_0(x)>0,\\
(u_0\cdot n, H_0\cdot n, J_0\cdot n)=(0,0,0),\ x\in \partial \Omega,\\
(\mathrm{curl}u_0\times n, \curl H_0\times n, \curl J_0\times n)=\left(-Au_0,0,-\frac{1+\delta}{2}(\psi_0\cdot n)J_0\right),\ x\in \partial \Omega,\\
(u_0\cdot n, H_0, J_0)=(0,0,0),\quad \mathrm{curl}u\times n=0,\ x\in \partial B_R.
\end{cases}
\end{equation}
Then there exist a small time $T_R$ and a unique classical solution $(\rho, u, H, J)$ to the following IBVP
\begin{equation}\label{9}
\left\{ \begin{array}{l}
\rho_t+\mathrm{div}(\rho u)=0,\\
\rho^{1-\delta}u_t+\rho^{1-\delta}u\cdot\nabla u-\mathcal{L}u+\frac{a\gamma}{\gamma-\delta}\nabla\rho^{\gamma-\delta}=\delta\psi\cdot \mathcal{S}(u)+ \rho^\frac{1-\delta}{2}J\cdot\nabla H-\rho^\frac{1-\delta}{2}\nabla H\cdot J,\\
H_{t}- \eta \Delta H=\curl(u\times H),\quad \div H=0,\\
J_t+u\cdot\nabla J-\eta\Delta J=J\cdot\nabla u+\frac{\delta-1}{2}J\div u+\frac{1+\delta}{2}\eta J\div\psi+\frac{(1+\delta)^2}{4}\eta J\psi^2\\
\quad+(1+\delta)\eta\psi\cdot\nabla J,\\
(u\cdot n, H\cdot n, J\cdot n)=(0,0,0),\ x\in \partial \Omega,\\
(\mathrm{curl}u\times n, \curl H\times n, \curl J\times n)=\left(-Au,0,-\frac{1+\delta}{2}(\psi\cdot n)J\right),\ x\in \partial \Omega,\\
(u\cdot n, H, J)=(0,0,0),\quad \mathrm{curl}u\times n=0,\ x\in \partial B_R,\\
(\rho, u, H, J)(x, 0) = (\rho_0, u_0, H_0, J_0)(x),\ x\in \Omega_R,
\end{array} \right.
\end{equation}
on $\Omega_R\times(0,T_R)$ such that
\begin{equation}
\left\{ \begin{array}{l}
\rho\in C([0,T_R]; H^3),\ \nabla\log\rho\in C([0,T_R]; H^2),\\
u,\ H\in C([0,T_R]; H^3)\cap L^2(0, T_R; D^1\cap D^4),\\
J\in C([0,T_R]; H^2)\cap L^2(0, T_R; D^1\cap D^3),\\
u_t,\ H_t \in L^\infty(0, T_R; H_0^1)\cap L^2(0, T_R; H^2),\ J_t \in L^\infty(0, T_R; L^2)\cap L^2(0, T_R; H^1),\\ \rho^\frac{1-\delta}{2}u_{tt},\ H_{tt} \in L^2(0, T_R; L^2),\\
\sqrt tu,\ \sqrt tH\in L^\infty(0, T_R; D^4),\ \sqrt tJ\in L^\infty(0, T_R; D^3),\\
\sqrt tu_t,\ \sqrt tH_t \in L^\infty(0, T_R; D^2), \ \sqrt tJ_t \in L^\infty(0, T_R; D^1),\\
\sqrt tu_{tt},\ \sqrt tH_{tt} \in L^\infty(0, T_R; L^2)\cap L^2(0, T_R; H^1),\ \sqrt tJ_{tt} \in L^2(0, T_R; L^2),\\
tu_t,\ tH_t \in L^\infty(0, T_R; D^3),\ tJ_t \in L^\infty(0, T_R; D^2),\\
tu_{tt},\ tH_{tt} \in L^\infty(0, T_R; D^1) \cap L^2(0, T_R; D^2),\ tJ_{tt} \in L^\infty(0, T_R; L^2) \cap L^2(0, T_R; D^1),\\
t u_{ttt},\ t H_{ttt} \in L^2(0, T_R; L^2),\\
t^{3/2}u_{tt},\ t^{3/2}H_{tt} \in L^\infty(0, T_R; D^2),\ t^{3/2}J_{tt} \in L^\infty(0, T_R; D^1),\\
t^{3/2}u_{ttt},\ t^{3/2}H_{ttt} \in L^\infty(0, T_R; L^2)\cap L^2(0, T_R; H_0^1),\ t^{3/2}J_{ttt} \in L^2(0, T_R; L^2).
\end{array} \right.
\end{equation}
\end{lemma}

In the rest of this section, we always assume that $ (\rho, u, H, J)$ is a solution  to the  IBVP \eqref{9} in $\Omega_R \times [0, T_R]$, which is obtained by
Lemma \ref{l25}.
The main aim of this section is to derive the following key a priori estimate on $\mathcal{E}$ defined by
\begin{equation}
	\begin{split}\label{qaqa}
		\mathcal{E}(t)\triangleq&1+\|\rho^\frac{1-\delta}{2}u\|_{L^2}
		+\|\nabla u\|_{L^2}+\|H\|_{H^1}+\|J\|_{H^1}+\|\rho^\frac{1-\delta}{2}u_t\|_{L^2}\\
		&+\|H_t\|_{L^2}+\|\rho^{\gamma-\frac{1+\delta}{2}}\|_{{W^{1,6}\cap D^1\cap D^2}}
		+\|\nabla\rho^\frac{\delta-1}{2}\|_{L^{6}\cap D^{1}}+\|\rho^\frac{\delta-1}{2}\|_{L^{6}(\Omega_0)}.
	\end{split}
\end{equation}
\begin{proposition}\label{p1}
	Assume that $(\rho_0, u_0, H_0, J_0)$ satisfies \eqref{2.16}. Let $(\rho,u,H,J)$ be the solution to the IBVP \eqref{9} in $\Omega_R\times(0,T_R]$ obtained by Lemma \ref{l25}. Then there exist positive constants $T_0$ and $M$ both depending only on $a$, $\delta$, $\gamma$, $\mu$, $\lambda$, $\eta$, $A$, $R_0$, and $C_0$ such that
    \begin{equation}\label{e32}
\begin{split}
			\sup_{0\leq t\leq T_0}\big(\mathcal{E}(t)+\|\nabla^2u\|_{L^2}+\|\nabla^2H\|_{L^2}\big)&+\int_0^{T_0}\Big(\|\rho^\frac{\delta-1}{2}\mathcal{L}u\|_{H^1}^2
            +\|\nabla u_t\|_{L^2}^2+\|\nabla^3 u\|_{L^2}^2\\
            &\quad+\|\nabla H_t\|_{L^2}^2+\|\nabla^3 H\|_{L^2}^2+\|J_t\|_{L^2}^2+\|\nabla^2 J\|_{L^2}^2\Big) ds
			\leq M,
\end{split}
	\end{equation}
    where
    \begin{equation}
    	\begin{split}
    	C_0\triangleq &1+\|\rho_0^\frac{1-\delta}{2}u_0\|_{L^2}+\|\nabla u_0\|_{H^1}+\|H_0\|_{H^2}+\|J_0\|_{H^1}
    		+\|\rho_0^\frac{\delta-1}{2}\mathcal{L}u_0\|_{L^2}\\
    		&+\|\rho_0^{\gamma-\frac{1+\delta}{2}}\|_{{W^{1,6}\cap D^1\cap D^2}}+\|\nabla\rho_0^\frac{\delta-1}{2}\|_{L^{6}\cap D^{1}}+\|\rho_0^\frac{\delta-1}{2}\|_{L^{6}(\Omega_0)} 
    	\end{split}
    \end{equation}
    is bounded constant due to \eqref{1117}, \eqref{118}, and \eqref{260114}.
\end{proposition}

To prove Proposition 3.1,  whose proof will be postponed to the end of this section, we will establish some necessary a priori estimates in  Lemmas \ref{L3.1}-\ref{L3.6}.
we begin with the following standard energy estimate for $(\rho,u,H,J)$.
\begin{lemma}\label{L3.1}
Under the conditions of Proposition \ref{p1}, let $(\rho,u,H,J)$ be a smooth solution to the IBVP \eqref{9}.
Then there exist a $T_1=T_1(C_0)>0$ and a positive constant $\beta=\beta(\delta,\gamma)>1$ such that for all $t\in (0,T_1]$,
\begin{equation}\label{e34}
\begin{split}
&\sup_{0\leq s\leq t}\left(\|\rho^\frac{1-\delta}{2}u\|_{L^2}^2+\|H\|_{L^2}^2+\|J\|_{L^2}^2\right)+\int_0^t\left(\|\nabla u\|_{L^2}^2+\|\nabla H\|_{L^2}^2+\|\nabla J\|_{L^2}^2\right)ds\\
&\leq CC_0^\beta+C\int_0^t\mathcal{E}^\beta ds,
\end{split}
\end{equation}
where (and in what follows) $C$ denotes a generic positive constant depending only on $a$, $\delta$, $\gamma$, $\mu$, $\lambda$, $\eta$, $A$, $R_0$, and $C_0$.
\end{lemma}
\begin{proof}
First,  it follows from \eqref{qaqa} that
\begin{equation}\label{rho_sup}
\|\rho\|_{L^\infty}=\|\rho^{\gamma-\frac{1+\delta}{2}}\|_{L^\infty}^\frac{2}{2\gamma-1-\delta}\leq C\|\rho^{\gamma-\frac{1+\delta}{2}}\|_{W^{1,6}}^\frac{2}{2\gamma-1-\delta}\leq C\mathcal{E}^\beta,
\end{equation}
in which $\beta=\beta(\delta, \gamma)>1$ is allowed to change from line to line. Therefore, one has 
\begin{gather}
\|\psi\|_{L^6}=\|\nabla\log\rho\|_{L^6}\leq C\|\rho\|_{L^\infty}^\frac{1-\delta}2\|\nabla\rho^{\frac{\delta-1}{2}}\|_{L^{6}}\leq C\mathcal{E}^\beta.\label{qaqa1}
\end{gather}
 
Next, it follows from Lemma \ref{elliptic_estimate}, Gagliardo-Nirenberg inequality, \eqref{9}$_2$, \eqref{9}$_5$-\eqref{9}$_7$, \eqref{qaqa}, \eqref{rho_sup}, and \eqref{qaqa1} that
\begin{equation}\label{e39}
\begin{split}
\|\nabla^2u\|_{L^2}
&\leq C\Big(\|\rho^{1-\delta}u_t\|_{L^2}+\|\rho^{1-\delta}u\cdot\nabla u\|_{L^2}+\|\nabla\rho^{\gamma-\delta}\|_{L^2}+\|\psi\cdot S(u)\|_{L^2}\\
&\quad\quad\quad+\|\rho^\frac{1-\delta}{2}J|\nabla H|\|_{L^2}+\|\nabla u\|_{L^2}\Big)\\
&\leq C\Big(\|\rho\|_{L^\infty}^\frac{1-\delta}2\|\rho^\frac{1-\delta}{2}u_t\|_{L^2}
+\|\rho\|_{L^\infty}^{1-\delta}\|\nabla u\|_{L^2}^\frac32\|\nabla u\|_{H^1}^\frac12
+\|\rho\|_{L^\infty}^\frac{1-\delta}2\|\nabla\rho^{\gamma-\frac{1+\delta}{2}}\|_{L^2}\\
&\quad\quad\quad+\|\psi\|_{L^{6}}\|\nabla u\|_{L^2}^\frac12\|\nabla u\|_{H^1}^\frac12
+\|\rho\|_{L^\infty}^\frac{1-\delta}2\|J\|_{L^6}\|\nabla H\|_{L^3}+\|\nabla u\|_{L^2}\Big)\\
&\leq C\mathcal{E}^\beta+\frac{1}{2}\|\nabla^2u\|_{L^2}+\frac14\|\nabla^2 H\|_{L^2}.
\end{split}
\end{equation}
Similarly, it follows from Lemma \ref{elliptic_estimate}, Gagliardo-Nirenberg inequality, \eqref{9}$_3$, \eqref{9}$_5$-\eqref{9}$_7$, \eqref{qaqa}, \eqref{rho_sup}, and \eqref{qaqa1}  that
\begin{equation}
\begin{split}\label{qaqa2}
\|\nabla^2H\|_{L^2}&\leq C\big(\|H_t\|_{L^2}+\||u||\nabla H|\|_{L^2}+\||\nabla u||H|\|_{L^2}+\|\nabla H\|_{L^2}\big)\\
&\leq C\big(\|H_t\|_{L^2}+\|\nabla u\|_{L^2}\|\nabla H\|_{L^2}^\frac12\|\nabla H\|_{H^1}^\frac12+\|\nabla H\|_{L^2}\big)\\
&\leq C\mathcal{E}^\beta+\frac{1}{4}\|\nabla^2H\|_{L^2}.
\end{split}
\end{equation}
This combined with 
 \eqref{e39}   and  Gagliardo-Nirenberg inequality yields that
\begin{align}\label{qaqa3}
\|u\|_{L^\infty}+\|\nabla^2u\|_{L^2}+\|H\|_{L^\infty}+\|\nabla^2H\|_{L^2}\leq C\mathcal{E}^\beta.
\end{align}

Now, multiplying the mass equation $\eqref{9}_1$ by $(1-\delta)\rho^{-\delta}$ show that
\begin{equation}\label{33}
(\rho^{1-\delta})_t+u\cdot\nabla\rho^{1-\delta}+(1-\delta)\rho^{1-\delta}\mathrm{div}u=0.
\end{equation}
Adding $\eqref{9}_2$ multiplied by  $u$ and  \eqref{33} multiplied by  $\frac{|u|^2}{2}$ together, we obtain after integrating the resulting equality on $\Omega_R$ by parts that
\begin{equation}\label{3.4}
\begin{split}
&\frac{1}{2}\frac{d}{dt}\|\rho^\frac{1-\delta}{2} u\|_{L^2}^2+\mu\|\mathrm{curl} u\|_{L^2}^2+(2\mu+\lambda)\|\mathrm{div}u\|_{L^2}^2
+\mu\int_{\partial\Omega}u\cdot A\cdot udS \\
&\leq C\int\Big(\rho^{1-\delta}|\mathrm{div}u||u|^2+|\nabla\rho^{\gamma-\frac{1+\delta}{2}}||\rho^\frac{1-\delta}{2} u|
+|\nabla\rho^\frac{\delta-1}{2}||\nabla u||\rho^\frac{1-\delta}{2} u|+ \rho^\frac{1-\delta}{2}|J||\nabla H| | u|\Big)dx\\
&\leq C\Big(\|\rho\|_{L^\infty}^\frac{1-\delta}{2}\|\nabla u\|_{L^2}\|u\|_{L^\infty}\|\rho^\frac{1-\delta}{2} u\|_{L^2}
+\|\nabla\rho^{\gamma-\frac{1+\delta}{2}}\|_{L^2}\|\rho^\frac{1-\delta}{2} u\|_{L^2}\\
&\quad\quad+\|\nabla\rho^\frac{\delta-1}{2}\|_{L^{6}}\|\nabla u\|_{L^{3}}\|\rho^\frac{1-\delta}{2} u\|_{L^2}+\|\rho\|_{L^\infty}^\frac{1-\delta}2\|J\|_{L^3}\|\nabla H\|_{L^2}\| u\|_{L^6}\Big)\\
&\leq C\mathcal{E}^\beta,
\end{split}
\end{equation}
where one has used Lemma \ref{G-N}, \eqref{qaqa}, \eqref{rho_sup}, and \eqref{qaqa3}.
Using the trace theorem, the boundary term in \eqref{3.4} can be governed as
\begin{align}\label{trace1}
\int_{\partial\Omega}u\cdot A\cdot udS\leq C\|u\|^2_{H^1(\Omega_0)}  &\leq C(\Omega_0)\|\nabla u\|_{L^2}^2\leq C\mathcal{E}^\beta.
\end{align}
Integrating   \eqref{3.4} over $(0,t)$,   using  \eqref{qaqa} and \eqref{trace1}, it yields that 
\begin{equation}
\begin{split}\nonumber
\sup_{0\leq s\leq t}\|\rho^\frac{1-\delta}{2}u\|_{L^2}^2+\int_0^t\|\nabla u\|_{L^2}^2ds
\leq CC_0^\beta+C\int_0^t\mathcal{E}^\beta ds.
\end{split}
\end{equation}

Then, integrating $\eqref{9}_3$  multiplied by $H$  over $\Omega_R$ leads to
\begin{equation}\label{qaqa4u}
\begin{split}
&\frac{1}{2}\frac{d}{dt}\|H\|_{L^2}^2+\eta\|\curl H\|_{L^2}^2\\
&\leq C\int(|\nabla u||H|^2+|u||\nabla H||H|)dx\\
&\leq C\|\nabla u\|_{L^2}\|H\|_{L^3}\|H\|_{L^6}+\|u\|_{L^6}\|\nabla H\|_{L^2}\|H\|_{L^3}\\
&\leq C\mathcal{E}^\beta,
\end{split}
\end{equation}
where one has used Lemma \ref{G-N}, \eqref{qaqa}, and \eqref{qaqa3}.
Integrating \eqref{qaqa4u} over $(0,t)$ and using  \eqref{qaqa}, one has
\begin{equation}
\begin{split}\nonumber
\sup_{0\leq s\leq t}\|H\|_{L^2}^2+\int_0^t\|\nabla H\|_{L^2}^2ds
\leq CC_0^\beta+C\int_0^t\mathcal{E}^\beta ds.
\end{split}
\end{equation}

Finally, multiplying $\eqref{9}_4$ by $J$ and integrating by parts lead  to
\begin{equation}\label{qaqa4}
\begin{split}
&\frac{1}{2}\frac{d}{dt}\|J\|_{L^2}^2+\eta\|\curl J\|_{L^2}^2+\eta\|\div J\|_{L^2}^2-\eta\int_{\partial\Omega}(\mbox{curl}J\times n)\cdot JdS\\
&\leq C\int\Big(|u\cdot\nabla J\cdot J|+J^2|\nabla u|+J^2|\nabla \psi|+J^2\psi^2+|\psi\cdot\nabla J\cdot J|\Big)dx\\
&\leq C\Big(\|u\|_{L^6}\|\nabla J\|_{L^2}\|J\|_{L^3}+\|J\|_{L^6}\|J\|_{L^3}\|\nabla u\|_{L^2}+\| \rho\|_{L^\infty}^{1-\delta}\|J\|_{L^3}^2\|\nabla\rho^\frac{\delta-1}{2}\|_{L^6}^2\\
&\quad+\|\rho\|_{L^\infty}^\frac{1-\delta}{2}\|J\|_{L^4}^2\|\nabla^2\rho^\frac{\delta-1}{2}\|_{L^2}
+\|J\|_{L^3}^2\|\psi\|_{L^6}^2+\|\psi\|_{L^6}\|\nabla J\|_{L^2}\|J\|_{L^3}\Big)\\
&\leq C\mathcal{E}^\beta,
\end{split}
\end{equation}
where one has used Lemma \ref{G-N}, \eqref{qaqa}, \eqref{rho_sup}, and \eqref{qaqa1}. Using the boundary condition $\eqref{9}_5$-$\eqref{9}_6$, the boundary term in \eqref{qaqa4} could be handled as follows,
\begin{equation}\label{1qaqa4}
\begin{split}
&\int_{\partial\Omega_R}(\mbox{curl}J\times n)\cdot JdS\\
&=-\frac{1+\delta}{2}\int_{\partial\Omega_R}(\psi \cdot n)|J|^2dS
=-\frac{1+\delta}{2}\int_{\Omega_R}\mbox{div}\left(\psi|J|^2\right)dx\\
&=-\frac{1+\delta}{2}\int_{\Omega_R}|J|^2\mbox{div}\psi dx-(1+\delta)\int_{\Omega_R}\psi\cdot\nabla J\cdot J dx\\
&\leq C\| \rho\|_{L^\infty}^{1-\delta}\|J\|_{L^3}^2\|\nabla\rho^\frac{\delta-1}{2}\|_{L^6}^2
+C\|\rho\|_{L^\infty}^\frac{1-\delta}{2}\|J\|_{L^4}^2\|\nabla^2\rho^\frac{1-\delta}{2}\|_{L^2}+C\|\psi\|_{L^{6}}\|\nabla J\|_{L^2}\|J\|_{L^3}\\
&\leq C\mathcal{E}^\beta,
\end{split}
\end{equation}
where one has used Lemma \ref{G-N}, \eqref{qaqa}, and \eqref{rho_sup}.

Thus, one derives \eqref{e34} after integrating  \eqref{qaqa4} over $(0,t)$ and using \eqref{qaqa}, \eqref{1qaqa4}. 
The proof of Lemma \ref{L3.1} is finished.
\end{proof}

The next  lemma concerns the estimates on  $L^\infty_tL^2_x$-norm of  $(\nabla J, \rho^\frac{1-\delta}{2}u_t, H_t)$.

\begin{lemma}\label{L3.3}
Let $(\rho,u,H,J)$ and $T_1$ be as in Lemma \ref{L3.1}. Then for all $t\in (0,T_1]$,
\begin{equation}\label{e313}
\begin{split}
&\sup_{0\leq s\leq t}\Big(\|\nabla J\|_{L^2}^2+\|\rho^\frac{1-\delta}{2}u_t\|_{L^2}^2+\|H_t\|_{L^2}^2\Big)+\int_0^t\Big(\|J_t\|_{L^2}^2+\|\nabla u_t\|_{L^2}^2+\|\nabla H_t\|_{L^2}^2\Big)ds\\
&\leq CC_0^\beta+C\int_0^t\mathcal{E}^\beta ds.
\end{split}
\end{equation}
\end{lemma}
\begin{proof}
\textbf{(1)} First, we will derive the estimate on $\nabla J$ and $J_t$. One deduces from Lemma \ref{l24}, Gagliardo-Nirenberg inequality, \eqref{9}$_4$-\eqref{9}$_7$, \eqref{qaqa}, \eqref{rho_sup}, and \eqref{qaqa1} that
\begin{equation}
\begin{split}\nonumber
\|\nabla^2J\|_{L^2}&\leq C\Big(\|J_t\|_{L^2}+\|u\cdot\nabla J\|_{L^2}+\|J|\nabla u|\|_{L^2}+\| J\mbox{div}\psi\|_{L^2}+\| J\psi^2\|_{L^2}\\
&\quad+\|\psi\cdot\nabla J\|_{L^2}+\big(1+\|\psi\|^2_{L^6\cap D^{1}}\big)\|J\|_{H^1}\Big)\\
&\leq C\Big(\|J_t\|_{L^2}+\|u\|_{L^6}\|\nabla J\|_{L^2}^\frac12\|\nabla J\|_{H^1}^\frac12+\|\nabla J\|_{L^2}^\frac12\|\nabla J\|_{H^1}^\frac12\|\nabla u\|_{L^2}\\
&\quad+\| \rho\|_{L^\infty}^{1-\delta}\|J\|_{L^6}\|\nabla\rho^\frac{\delta-1}{2}\|_{L^6}^2
+\|\rho\|_{L^\infty}^\frac{1-\delta}{2}\|\nabla J\|_{L^2}^\frac12\|\nabla J\|_{H^1}^\frac12\|\nabla^2\rho^\frac{\delta-1}{2}\|_{L^2}\\
&\quad+\| J\|_{L^6}\|\psi\|_{L^6}^2+\|\psi\|_{L^6}\|\nabla J\|_{L^2}^\frac12\|\nabla J\|_{H^1}^\frac12+\mathcal{E}^\beta\Big)\\
&\leq C\mathcal{E}^\beta+C\|J_t\|_{L^2}+\frac{1}{4}\|\nabla^2J\|_{L^2},
\end{split}
\end{equation}
which directly yields that
\begin{align}\label{muy01}
\|\nabla^2J\|_{L^2}\leq C\mathcal{E}^\beta+C\|J_t\|_{L^2}.
\end{align}

Next, multiplying $\eqref{9}_4$ by $J_t$ and integrating the resulting equality by parts, it holds
\begin{align}\label{plm2}
&\frac{\eta}{2}\frac{d}{dt}\left(\|\mathrm{curl}J\|_{L^2}^2+\|\mathrm{div}J\|_{L^2}^2\right)+\| J_t\|_{L^2}^2-\eta\int_{\partial\Omega_R}(\mbox{curl}J\times n)\cdot J_tdS\nonumber\\
&\leq C\int\Big(|u\cdot\nabla J\cdot J_t|+|J||\nabla u|||J_t|+|J||\nabla \psi||J_t|+|J|\psi^2|J_t|+|\psi\cdot\nabla J\cdot J_t|\Big)dx\nonumber\\
&\leq C\Big(\|u\|_{L^\infty}\|\nabla J\|_{L^2}\|J_t\|_{L^2}+\|J\|_{L^\infty}\|\nabla u\|_{L^2}\|J_t\|_{L^2}+\| \rho\|_{L^\infty}^{1-\delta}\|J\|_{L^6}\|\nabla\rho^\frac{\delta-1}{2}\|_{L^6}^2\|J_t\|_{L^2}\nonumber\\
&\quad+\|\rho\|_{L^\infty}^\frac{1-\delta}{2}\|J\|_{L^\infty}\|\nabla^2\rho^\frac{\delta-1}{2}\|_{L^2}\|J_t\|_{L^2}
+\|J\|_{L^6}\|\psi\|_{L^6}^2\|J_t\|_{L^2}+\|\psi\|_{L^6}\|\nabla J\|_{L^3}\|J_t\|_{L^2}\Big)\nonumber\\
&\leq \frac{1}{2}\|J_t\|_{L^2}^2+C\mathcal{E}^\beta,
\end{align}
where one has used Lemma \ref{G-N}, \eqref{qaqa},  \eqref{rho_sup}, \eqref{qaqa1}, \eqref{qaqa3},  \eqref{muy01} and the following boundary estimates: 
\begin{align*}
-\eta\int_{\Omega_R}\Delta J\cdot J_tdx&=\eta\int_{\Omega_R}\mbox{curl}^2 J\cdot J_tdx-\eta\int_{\Omega_R}\nabla
\mbox{div} J\cdot J_tdx\\
&=\frac{\eta}{2}\frac{d}{dt}\left(\int_{\Omega_R}|\mbox{curl}J|^2dx+\int_{\Omega_R}|\mbox{div}J|^2dx\right)-\eta\int_{\partial\Omega_R}(\mbox{curl}J\times n)\cdot J_tdS.
\end{align*} 
Note that
\begin{equation}\label{rho_lb}
    \|\rho^{-\frac{1+\delta}{2}}\|_{L^\infty(\Omega_0)}
    =\|\rho^\frac{\delta-1}{2}\|_{L^\infty(\Omega_0)}^\frac{1+\delta}{1-\delta}
    \leq C\|\rho^\frac{\delta-1}{2}\|_{W^{1,6}(\Omega_0)}^\frac{1+\delta}{1-\delta}
    \leq C\mathcal{E}^\beta,
\end{equation}
which together with the boundary conditions  $\eqref{9}_5$-$\eqref{9}_6$ yields
\begin{equation}\label{ze313}
\begin{split}
&\int_{\partial\Omega_R}(\mbox{curl}J\times n)\cdot J_tdS=-\frac{1+\delta}{2}\int_{\partial\Omega_R}(\psi \cdot n)J \cdot J_tdS\\
&=-\frac{1+\delta}{2}\int_{\partial\Omega}(n\cdot\psi) J \cdot \Big(H_t\rho^{-\frac{1+\delta}{2}}-Hu\cdot\nabla\rho^{-\frac{1+\delta}{2}}+\frac{1+\delta}{2}H\rho^{-\frac{1+\delta}{2}}\div u\Big)dx\\
&\leq C\mathcal{E}^\beta\left\||\psi| |H|\big(|H_t|+|Hu\cdot\psi|+|H\div u|\big)\right\|_{W^{1,1}(\Omega_0)}\\
&\leq C\mathcal{E}^\beta\bigl(\|\psi\|_{L^6}\|H\|_{L^2}\|H_t\|_{L^3}+\|\rho\|_{L^\infty}^{1-\delta}\|H\|_{L^6}\|H_t\|_{L^2}\|\nabla\rho^\frac{\delta-1}{2}\|_{L^6}^2
\\
&\quad+\|\rho\|_{L^\infty}^\frac{1-\delta}{2}\|H\|_{L^\infty}\|H_t\|_{L^2}\|\nabla^2\rho^\frac{1-\delta}{2}\|_{L^2}+\|\psi\|_{L^6}\|\nabla H\|_{L^2}\|H_t\|_{L^3}\\
&\quad+\|\psi\|_{L^6}\|H\|_{L^3}\|\nabla H_t\|_{L^2}+\|\psi\|_{L^6}^2\|H\|_{L^4}^2\|u\|_{L^6}\\
&\quad+\|\rho\|_{L^\infty}^{1-\delta}\|\nabla\rho^\frac{\delta-1}{2}\|_{L^6}^2\|\psi\|_{L^6}\|H\|_{L^4}^2\|u\|_{L^\infty}
+\|\rho\|_{L^\infty}^\frac{1-\delta}{2}\|\nabla^2\rho^\frac{\delta-1}{2}\|_{L^2}\|\psi\|_{L^6}\|H\|_{L^6}^2\|u\|_{L^\infty}\\
&\quad+\|\psi\|_{L^6}^2\|\nabla H\|_{L^2}\|H\|_{L^6}\|u\|_{L^\infty}+\|\psi\|_{L^6}^2\|H\|_{L^4}^2\|\nabla u\|_{L^6}+\|\psi\|_{L^6}\|H\|_{L^6}^2\|\nabla u\|_{L^2}\\
&\quad+\|\rho\|_{L^\infty}^{1-\delta}\|\nabla\rho^\frac{\delta-1}{2}\|_{L^6}^2\|H\|_{L^4}^2\|\nabla u\|_{L^6}+\|\rho\|_{L^\infty}^\frac{1-\delta}{2}\|\nabla^2\rho^\frac{\delta-1}{2}\|_{L^2}\|H\|_{L^6}^2\|\nabla u\|_{L^6}\\
&\quad+\|\psi\|_{L^6}\|\nabla H\|_{L^6}\|H\|_{L^6}\|\nabla u\|_{L^2}+\|\psi\|_{L^6}\|H\|_{L^6}^2\|\nabla^2 u\|_{L^2}\bigr)\\
&\leq C\mathcal{E}^\beta+\epsilon\|\nabla H_t\|_{L^2}^2.
\end{split}
\end{equation}
This along with   \eqref{plm2} yields 
\begin{equation}\label{muyo}
\eta\frac{d}{dt}\left(\|\mathrm{curl}J\|_{L^2}^2+\|\mathrm{div}J\|_{L^2}^2\right)+\| J_t\|_{L^2}^2\leq C\mathcal{E}^\beta+\epsilon\|\nabla H_t\|_{L^2}^2.
	\end{equation}

\textbf{(2)} Now, we will derive the estimates on $\rho^\frac{1-\delta}{2}u_t$ and $\nabla u_t$. It follows from $\eqref{9}_1$ that $\rho^\frac{1-\delta}{2}$, $\rho^{\gamma-\delta}$ and $\psi$  satisfy
$$
\begin{gathered}
(\rho^\frac{1-\delta}{2})_t=-u\cdot\nabla\rho^\frac{1-\delta}{2}-\frac{1-\delta}{2}\rho^\frac{1-\delta}{2}\mathrm{div}u,\\
(\rho^{\gamma-\delta})_t=-u\cdot\nabla\rho^{\gamma-\delta}-(\gamma-\delta)\rho^{\gamma-\delta}\mathrm{div}u,\\
\psi_t=-u\cdot\nabla\psi-\nabla u\cdot \psi-\nabla\mathrm{div}u,
\end{gathered}
$$
which combining with Lemma \ref{G-N}, \eqref{qaqa},  \eqref{rho_sup}, \eqref{qaqa1}, and \eqref{qaqa3} that
\begin{gather}
\|(\rho^\frac{1-\delta}{2})_t\|_{L^3}\leq\|\rho\|_{L^\infty}^{1-\delta}\|u\|_{L^6}\|\nabla\rho^{\frac{\delta-1}{2}}\|_{L^6}+\|\rho\|_{L^\infty}^{(1-\delta)/2}\|\nabla u\|_{L^3}\leq C\mathcal{E}^\beta,\label{qaz1}\\
\|(\rho^{\gamma-\delta})_t\|_{L^2}\leq\|\rho\|_{L^\infty}^{(1-\delta)/2}\|u\|_{L^\infty}\|\nabla\rho^{\gamma-\frac{1+\delta}{2}}\|_{L^2}+\|\rho\|_{L^\infty}^{\gamma-\delta}\|\nabla u\|_{L^2}\leq C\mathcal{E}^\beta,\label{qaz2}
\end{gather}
\begin{equation}\label{qaz4}
\begin{split}
\|\psi_t\|_{L^2}
&\leq C\bigl(\|\nabla u\|_{L^{3}}\|\psi\|_{L^{6}}+\|\rho\|_{L^\infty}^{1-\delta}\|u\|_{L^{6}}\|\nabla\rho^\frac{\delta-1}{2}\|_{L^{6}}^2\\
&\quad+\|\rho\|_{L^\infty}^{(1-\delta)/2}\|u\|_{L^\infty}\|\nabla^2\rho^\frac{\delta-1}{2}\|_{L^{2}}
+\|\nabla\mathrm{div}u\|_{L^2}\bigr)\\
&\leq C\mathcal{E}^\beta.
\end{split}
\end{equation}

Next, operating $\partial_t$ to $\eqref{9}_2$ yields that
\begin{equation}\label{32}
\begin{split}
&(\rho^{1-\delta})_tu_t+\rho^{1-\delta}u_{tt}
+(\rho^{1-\delta})_tu\cdot\nabla u
+\rho^{1-\delta}u_t\cdot\nabla u
+\rho^{1-\delta}u\cdot\nabla u_t
-\mathcal{L}u_t
+\frac{a\gamma}{\gamma-\delta}\nabla(\rho^{\gamma-\delta})_t\\
&=\delta\psi_t\cdot S(u)+\delta\psi\cdot S(u)_t+ (\rho^\frac{1-\delta}{2})_tJ\cdot\nabla H+\rho^\frac{1-\delta}{2}J_t\cdot\nabla H+ \rho^\frac{1-\delta}{2}J\cdot\nabla H_t\\
&\quad-(\rho^\frac{1-\delta}{2})_t\nabla H\cdot J-\rho^\frac{1-\delta}{2}\nabla H_t\cdot J-\rho^\frac{1-\delta}{2}\nabla H\cdot J_t.
\end{split}
\end{equation}
Multiplying the above equality by $u_t$, we obtain after using integration by parts and $\eqref{33}$ that
\begin{equation}
\begin{split}\label{qaqa11}
&\frac{1}{2}\frac{d}{dt}\|\rho^\frac{1-\delta}{2}u_t\|_{L^2}^2+\mu\|\mathrm{curl} u_t\|_{L^2}^2+(2\mu+\lambda)\|\mathrm{div}u_t\|_{L^2}^2
+\mu\int_{\partial\Omega}u_t\cdot A\cdot u_tdS \\
&\leq C\int\Big(\rho^{1-\delta}|\nabla u||u_t|^2+\rho^{1-\delta}|u||\nabla u_t||u_t|+|u|^2|\nabla\rho^{1-\delta}||\nabla u||u_t|+\rho^{1-\delta}|u||\nabla u|^2|u_t|\\
&\quad+|(\rho^{\gamma-\delta})_t||\div u_t|+|\psi_t||\nabla u||u_t|+|\psi||\nabla u_t||u_t|+|(\rho^\frac{1-\delta}{2})_t||J||\nabla H||u_t|\\
&\quad+\rho^\frac{1-\delta}{2}|J_t||\nabla H||u_t|+\rho^\frac{1-\delta}{2}|J||\nabla H_t||u_t|\Big)dx
\triangleq\sum_{i=1}^{10} I_i.
\end{split}
\end{equation}
Thanks to Lemma \ref{G-N}, \eqref{qaqa}, \eqref{rho_sup}, \eqref{qaqa1}, \eqref{qaqa3}, and \eqref{qaz1}-\eqref{qaz4}, we can estimate each $I_i$ as follows:
$$
\begin{aligned}
&I_1\leq C\|\rho\|_{L^\infty}^{(1-\delta)/2}\|\nabla u\|_{L^3}\|\rho^\frac{1-\delta}{2}u_t\|_{L^2}\|u_t\|_{L^6}
\leq C\mathcal{E}^\beta+\epsilon\|\nabla u_t\|_{L^2}^2,\\
&I_2\leq C\|\rho\|_{L^\infty}^{(1-\delta)/2}\|u\|_{L^\infty}\|\nabla u_t\|_{L^2}\|\rho^\frac{1-\delta}{2}u_t\|_{L^2}
\leq C\mathcal{E}^\beta+\epsilon\|\nabla u_t\|_{L^2}^2,\\
&I_3\leq C\|\rho\|_{L^\infty}^{3(1-\delta)/2}\|u\|_{L^6}\|u\|_{L^\infty}\|\nabla\rho^\frac{\delta-1}{2}\|_{L^{6}}\|\nabla u\|_{L^2}\|u_t\|_{L^6}
\leq C\mathcal{E}^\beta+\epsilon\|\nabla u_t\|_{L^2}^2,\\
&I_4\leq C\|\rho\|_{L^\infty}^{1-\delta}\|u\|_{L^6}\|\nabla u\|_{L^2}\|\nabla u\|_{L^6}\|u_t\|_{L^6}
\leq C\mathcal{E}^\beta+\epsilon\|\nabla u_t\|_{L^2}^2,\\
&I_5\leq C\|(\rho^{\gamma-\delta})_t\|_{L^2}\|\nabla u_t\|_{L^2}
\leq C\mathcal{E}^\beta+\epsilon\|\nabla u_t\|_{L^2}^2,\\
&I_6\leq C\|\psi_t\|_{L^2}\|\nabla u\|_{L^3}\|u_t\|_{L^6}
\leq C\mathcal{E}^\beta+\epsilon\|\nabla u_t\|_{L^2}^2,\\
&I_7 \leq C\|\nabla \rho^\frac{\delta-1}{2}\|_{L^{6}}\|\nabla u_t\|_{L^2}\|\rho^\frac{1-\delta}{2}u_t\|_{L^3}\leq C\mathcal{E}^\beta+\epsilon\|\nabla u_t\|_{L^2}^2,\\
&I_{8}\leq C\|(\rho^\frac{1-\delta}{2})_t\|_{L^3}\|J\|_{L^6}\|\nabla H\|_{L^{3}}\|u_t\|_{L^6}\leq C\mathcal{E}^\beta+\epsilon\|\nabla u_t\|_{L^2}^2,\\
&I_{9}\leq C\|J_t\|_{L^2}\|\nabla H\|_{L^{6}}\|\rho^\frac{1-\delta}{2}u_t\|_{L^3}\leq C\mathcal{E}^\beta+\epsilon\|J_t\|_{L^2}^2+\epsilon\|\nabla u_t\|_{L^2}^2,\\
&I_{10}\leq C\|J\|_{L^6}\|\nabla H_t\|_{L^{2}}\|\rho^\frac{1-\delta}{2}u_t\|_{L^3}\leq C\mathcal{E}^\beta+\epsilon\|\nabla u_t\|_{L^2}^2+\epsilon\|\nabla H_t\|_{L^2}^2.
\end{aligned}
$$

Using the trace theorem, the boundary term in \eqref{qaqa11} can be governed as
\begin{equation}\label{trace3}
\begin{split}
\int_{\partial\Omega}u_t\cdot A\cdot u_tdS
&\leq C\left\||u_t|^2\right\|_{W^{1,1}(\Omega_0)}  \\
&\leq C\|\rho^{\delta-1}\|_{L^\infty(\Omega_0)}\|\rho^\frac{1-\delta}2u_t\|_{L^2(\Omega_0)}^2+C\|\rho^\frac{\delta-1}2\|_{L^\infty(\Omega_0)}\|\rho^\frac{1-\delta}2u_t\|_{L^2}\|\nabla u_t\|_{L^2}\\
&\leq C\mathcal{E}^\beta\|\rho^\frac{1-\delta}{2}u_t\|_{L^2}^2+\epsilon\|\nabla u_t\|_{L^2}^2
\end{split}
\end{equation}
owing to \eqref{rho_lb}.
And, it follows from  Lemma \ref{l22}, \eqref{qaqa}, and \eqref{rho_lb} that
\begin{equation}\label{3.4aqa0}
\begin{split}
\|\nabla u_t\|_{L^2}
&\leq C\big(\|\curl u_t\|_{L^2}+\|\div u_t\|_{L^2}+\|u_t\|_{L^2(\Omega_0)}\big)\\
&\leq C\big(\|\curl u_t\|_{L^2}+\|\div u_t\|_{L^2}+\|\rho^\frac{\delta-1}2\|_{L^\infty(\Omega_0)}\|\rho^\frac{1-\delta}2u_t\|_{L^2(\Omega_0)}\big)\\
&\leq C\big(\|\curl u_t\|_{L^2}+\|\div u_t\|_{L^2}+\mathcal{E}^\beta\big).
\end{split}
\end{equation}

Substituting $I_1$--$I_{10}$ into \eqref{qaqa11}, combining with \eqref{trace3} and \eqref{3.4aqa0}, choosing $\epsilon$ small enough, one obtains
\begin{equation}
\frac{d}{dt}\|\rho^\frac{1-\delta}{2}u_t\|_{L^2}^2+C\|\nabla u_t\|_{L^2}^2\leq C\mathcal{E}^\beta+\epsilon\|J_t\|_{L^2}^2+\epsilon\|\nabla H_t\|_{L^2}^2. \label{35}
\end{equation}

\textbf{(3)} Then, we will handle the estimates on $\nabla H$ and $H_t$. Operating $\partial_t$ to $\eqref{9}_3$ yields that
\begin{equation}
\begin{split}\nonumber
&H_{tt}-\eta\Delta H_t=\curl(u_t\times H)+\curl(u\times H_t).
\end{split}
\end{equation}
Multiplying the above equality by $H_t$, we obtain after using integration by parts that
\begin{equation}
\begin{split}\label{pop11}
&\frac12\frac{d}{dt}\|H_{t}\|_{L^2}^2+C\|\nabla H_t\|_{L^2}^2\\
&\leq C\int\Big(|\nabla u_t||H||H_t|+|u_t||\nabla H||H_t|+|\nabla u||H_t||H_t|+|u||\nabla H_t||H_t|\Big)dx\\
&\leq C\|\nabla u_t\|_{L^2}\|H\|_{L^\infty}\|H_t\|_{L^2}+C\|u_t\|_{L^6}\|\nabla H\|_{L^3}\|H_t\|_{L^2}+C\|\nabla u\|_{L^3}\|H_t\|_{L^2}\|H_t\|_{L^6}\\
&\quad +C\|u\|_{L^\infty}\|\nabla H_t\|_{L^2}\|H_t\|_{L^2}\\
&\leq C\mathcal{E}^\beta+\epsilon\|\nabla u_t\|_{L^2}^2+\epsilon\|\nabla H_t\|_{L^2}^2,
\end{split}
\end{equation}
where one has used Lemma \ref{G-N}, \eqref{qaqa}, \eqref{qaqa3} and
\begin{align}\label{hqaz}
\|\nabla H_t\|_{L^2}\leq C\|\curl H_t\|_{L^2}+C\|H_t\|_{L^2}
\leq C\|\curl H_t\|_{L^2}+C\mathcal{E}^\beta.
\end{align}
%

\textbf{(4)} With  \eqref{muyo}, \eqref{35} and \eqref{pop11} at hand, add them together and choosing $\epsilon$ small enough, we have
\begin{equation}
\begin{split}
&\frac{d}{dt}\Big(\eta(\|\div J\|_{L^2}^2+\|\curl J\|_{L^2}^2)+\|\rho^\frac{1-\delta}{2}u_t\|_{L^2}^2+\|H_t\|_{L^2}^2\Big)+\| J_t\|_{L^2}^2+C\|\nabla u_t\|_{L^2}^2\\
&\quad+C\|\nabla H_t\|_{L^2}^2\leq C\mathcal{E}^\beta. \label{35q}
\end{split}
\end{equation}

Multiplying the reformulated momentum equation $\eqref{9}_2$ by $u_t$, we can easily deduce that
\begin{equation}\label{317}
\begin{split}
&\limsup\limits_{\tau \to 0}\|\rho^\frac{1-\delta}{2}u_t(\tau)\|_{L^2}\\
&\leq C\Big(\|\rho_0\|_{L^\infty}^\frac{1-\delta}{2}\|u_0\|_{L^\infty}\|\nabla u_0\|_{L^2}
+\|\rho_0^\frac{\delta-1}{2}\mathcal{L}u_0\|_{L^2}+\|\nabla\rho_0^{\gamma-\frac{1+\delta}{2}}\|_{L^2}
\\
&\quad+\|\nabla\rho_0^\frac{\delta-1}{2}\|_{L^{6}}\|\nabla u_0\|_{H^1}+\|J_0\|_{L^6}\|\nabla H_0\|_{L^3}\Big)\\
&\leq C\bigl(\|\rho_0^{\gamma-\frac{1+\delta}{2}}\|_{W^{1,6}}^\frac{1-\delta}{2\gamma-1-\delta}\|\nabla u_0\|_{H^1}\|\nabla u_0\|_{L^2}
+\|\rho_0^\frac{\delta-1}{2}\mathcal{L}u_0\|_{L^2}
+\|\nabla\rho_0^{\gamma-\frac{1+\delta}{2}}\|_{L^2}\\
&\quad+\|\nabla\rho_0^\frac{\delta-1}{2}\|_{L^{6}}\|\nabla u_0\|_{H^1}
+\|\nabla J_0\|_{L^2}\|\nabla H_0\|_{H^1}\bigr)
\\
&\leq CC_0^\beta.
\end{split}
\end{equation}

From the magnetic field equations \eqref{9}$_3$, we have
\begin{equation}
\begin{split}\label{hczg}
\limsup\limits_{\tau \to 0}\|H_t(\tau)\|_{L^2}
&\leq C(\|\nabla u_0\|_{L^2}\|H_0\|_{L^\infty}+\|u_0\|_{L^6}\|\nabla H_0\|_{L^3}+\|\nabla^2 H_0\|_{L^2})\\
&\leq C(\|\nabla u_0\|_{L^2}\|H_0\|_{H^2}+\|\nabla u_0\|_{L^2}\|\nabla H_0\|_{H^1}+\|\nabla^2 H_0\|_{L^2})\\
&\leq CC_0^\beta.
\end{split}
\end{equation}

Finally, integrating \eqref{35q} over $(\tau,t)$ and letting $\tau\rightarrow 0$, we obtain \eqref{e313} after using Lemma \ref{l22}, \eqref{e34}, \eqref{317}, and \eqref{hczg} . The proof of Lemma \ref{L3.3} is completed.
\end{proof}

Next, the following lemma concerns the estimates on  $L^\infty_tL^2_x$-norm of  $(\nabla u, \nabla H)$.


\begin{lemma}\label{L3.2}
	Let $(\rho,u,H,J)$ and $T_1$ be as in Lemma \ref{L3.1}. Then for all $t\in (0,T_1]$,
	\begin{equation}\label{e310}
    \begin{split}
		&\sup_{0\leq s\leq t}\Big(
        \|\nabla u\|_{L^2}^2+\|\nabla H\|_{L^2}^2\Big)+\int_0^t\Big(\|\rho^\frac{\delta-1}{2}\mathcal{L}u\|_{L^2}^2+\|H_t\|_{L^2}\Big)ds\\
		&\leq CC_0^\beta+C\int_0^t\mathcal{E}^\beta ds.
        \end{split}
	\end{equation}
\end{lemma}
\begin{proof}
First, multiplying $\eqref{9}_2$ by $-2\rho^{\delta-1}\mathcal{L}u$ and integrating the resulting equality by parts leads to
\begin{equation}
\begin{split}\label{za3.4}
&\frac{d}{dt}\left(\mu\|\mathrm{curl}u\|_{L^2}^2+(2\mu+\lambda)\|\mathrm{div}u\|_{L^2}^2\right)
+2\|\rho^\frac{\delta-1}{2}\mathcal{L}u\|_{L^2}^2+2\mu\int_{\partial\Omega} u\cdot A\cdot u_tdS\\
&\leq C\int\Big(\rho^\frac{1-\delta}{2}|u||\nabla u|+|\nabla\rho^{\gamma-\frac{1+\delta}{2}}|
+|\nabla\rho^\frac{\delta-1}{2}||\nabla u|+|J||\nabla H|\Big)|\rho^\frac{\delta-1}{2}\mathcal{L}u|dx\\
&\leq \|\rho^\frac{\delta-1}{2}\mathcal{L}u\|_{L^2}^2+C\Big(\|\rho\|_{L^\infty}^{1-\delta}\|u\|_{L^\infty}^2\|\nabla u\|_{L^2}^2
+\|\nabla\rho^{\gamma-\frac{1+\delta}{2}}\|_{L^2}^2+\|\nabla\rho^\frac{\delta-1}{2}\|_{L^{6}}^2\|\nabla u\|_{L^3}^2\\
&\quad+\|J\|_{L^6}^2\|\nabla H\|_{L^3}^2\Big)\\
&\leq \|\rho^\frac{\delta-1}{2}\mathcal{L}u\|_{L^2}^2+C\mathcal{E}^\beta,
\end{split}
\end{equation}
where we used Lemma \ref{G-N}, \eqref{qaqa}, \eqref{rho_sup}, and \eqref{qaqa3}.
Similar to \eqref{trace1}, the boundary term on the right hand of \eqref{za3.4} can be governed as
\begin{align}\label{trace2}
\int_{\partial\Omega} u\cdot A\cdot u_tdS\leq C\||u||u_t|\|_{W^{1,1}(\Omega_0)}  &\leq C\mathcal{E}^\beta+C\|\nabla u_t\|_{L^2}^2.
\end{align}
Integrating \eqref{za3.4} over $(0,t)$, together with  \eqref{trace2} and \eqref{e313} yields 
\begin{equation}
\begin{split}\nonumber
\sup_{0\leq s\leq t}\big(\|\mathrm{curl}u\|_{L^2}^2+\|\mathrm{div}u\|_{L^2}^2\big)+\int_0^t\|\rho^\frac{\delta-1}{2}\mathcal{L}u\|_{L^2}^2ds
\leq CC_0^\beta+C\int_0^t\mathcal{E}^\beta ds.
\end{split}
\end{equation}

Next, we multiply $\eqref{9}_3$ by $2H_t$ and integrate the resulting equality by parts, then the following holds
$$
\begin{aligned}
&\eta\frac{d}{dt}\|\mathrm{curl}H\|_{L^2}^2+2\| H_t\|_{L^2}^2\\
&\leq C\int(|\nabla u||H||H_t|+|u||\nabla H||H_t|)dx\\
&\leq C\|\nabla u\|_{L^3}\|H\|_{L^6}\|H_t\|_{L^2}+\|u\|_{L^6}\|\nabla H\|_{L^3}\|H_t\|_{L^2}\\
&\leq C\mathcal{E}^\beta,
\end{aligned}
$$
where we used Lemma \ref{G-N}, \eqref{qaqa}, and \eqref{qaqa3}.
Integrating the resulting inequality over $(0,t)$, we obtain \eqref{e310} after using \eqref{e34} and Lemma \ref{l22}.
The proof of Lemma \ref{L3.2} is completed.
\end{proof}
The following lemma is concerning the  estimate of $\rho$ with positive power $\gamma-\frac{1+\delta}{2}>0$.
\begin{lemma}\label{L3.5}
	Let $(\rho,u,H,J)$ and $T_1$ be as in Lemma \ref{L3.1}. Then for all $t\in (0,T_1]$,
	\begin{equation}\label{e336}
		\sup_{0\leq s\leq t} \|\rho^{\gamma-\frac{1+\delta}{2}}\|_{W^{1,6}\cap D^1\cap D^2}\leq CC_0^\beta+C\int_0^t\mathcal{E}^\beta ds.
	\end{equation}
\end{lemma}
\begin{proof}
First, it follows from Lemma \ref{G-N}, \ref{elliptic_estimate}, \eqref{9}$_2$, \eqref{9}$_5$-\eqref{9}$_7$, \eqref{rho_sup}, \eqref{qaqa1}, and \eqref{qaqa3} that
\begin{equation}\nonumber
\begin{split}
&\|\nabla^3u\|_{L^2}\leq C\|\mathcal{L}u\|_{H^1}+C\|\nabla u\|_{L^2}\\
&\leq C\Big(\|\nabla\rho^\frac{1-\delta}{2}\|_{L^{6}}\|\rho^\frac{1-\delta}{2}u_t\|_{L^{3}}
+\|\rho^{1-\delta}\|_{L^\infty}\|\nabla u_t\|_{L^2}+\|\nabla\rho^{1-\delta}\|_{L^{6}}\|u\|_{L^6}\|\nabla u\|_{L^{6}}\\
&\quad+\|\rho^{1-\delta}\|_{L^\infty}\|\nabla u\|_{L^3}\|\nabla u\|_{L^6}
+\|\rho^{1-\delta}\|_{L^\infty}\|u\|_{L^\infty}\|\nabla^2u\|_{L^2}+\|\nabla^2\rho^{\gamma-\delta}\|_{L^2}\\
&\quad+\|\rho^\frac{1-\delta}{2}\|_{L^\infty}\|\nabla^2\rho^\frac{\delta-1}{2}\|_{L^{2}}\|\nabla u\|_{L^{\infty}}+\|\nabla\rho^\frac{1-\delta}{2}\|_{L^{6}}\|\nabla\rho^\frac{\delta-1}{2}\|_{L^{6}}\|\nabla u\|_{L^{6}}\\
&\quad+\|\rho^\frac{1-\delta}{2}\|_{L^\infty}\|\nabla\rho^\frac{\delta-1}{2}\|_{L^{6}}\|\nabla^2u\|_{L^{3}}
+\|\nabla\rho^\frac{1-\delta}{2}\|_{L^{6}}\|J\|_{L^6}\|\nabla H\|_{L^6}\\
&\quad+\|\rho^\frac{1-\delta}{2}\|_{L^{\infty}}\|\nabla J\|_{L^3}\|\nabla H\|_{L^6}+\|\rho^\frac{1-\delta}{2}\|_{L^{\infty}}\|J\|_{L^\infty}\|\nabla^2H\|_{L^{2}}\Big)+C\mathcal{E}^\beta\\
&\leq C\Big(\|\rho\|_{L^\infty}^{(1-\delta)5/4}
\|\nabla\rho^{\frac{\delta-1}{2}}\|_{L^{6}}
\|\rho^\frac{1-\delta}{2}u_t\|_{L^2}^{1/2}\|\nabla u_t\|_{L^2}^{1/2}+\|\rho\|_{L^\infty}^{1-\delta}\|\nabla u_t\|_{L^2}\\
&\quad+\|\rho\|_{L^\infty}^{3(1-\delta)/2}\|\nabla\rho^{\frac{\delta-1}{2}}\|_{L^{6}}
\|\nabla u\|_{H^1}^2+\|\rho\|_{L^\infty}^{1-\delta}
\|\nabla u\|_{H^1}^2
+\|\rho\|_{L^\infty}^{(1-\delta)/2}\|\nabla^2\rho^{\gamma-\frac{1+\delta}{2}}\|_{L^2}\\ &\quad+\|\rho\|_{L^\infty}^{1-\delta}\|\nabla\rho^{\gamma-\frac{1+\delta}{2}}\|_{L^2}\|\nabla\rho^{\frac{\delta-1}{2}}\|_{L^{6}}+\|\rho\|_{L^\infty}^{(1-\delta)/2}
\|\nabla\rho^{\frac{\delta-1}{2}}\|_{L^{6}\cap D^{1}}\|\nabla u\|_{H^1}^{1/2}\|\nabla u\|_{H^2}^{1/2}\\
&\quad+\|\rho\|_{L^\infty}^{1-\delta}\|\nabla\rho^{\frac{\delta-1}{2}}\|_{L^{6}}^2\|\nabla u\|_{H^1}
+\|\rho^{1-\delta}\|_{L^{\infty}}\|\nabla\rho^\frac{\delta-1}{2}\|_{L^{6}}\|\nabla J\|_{L^2}\|\nabla H\|_{H^1}\\
&\quad+\|\rho^\frac{1-\delta}{2}\|_{L^{\infty}}\|\nabla J\|_{L^2}^{1/2}\|\nabla J\|_{H^1}^{1/2}\|\nabla H\|_{H^1}
\Big)+C\mathcal{E}^\beta\\
&\leq C\mathcal{E}^\beta+C\mathcal{E}^\beta\|\nabla u_t\|_{L^2}+C\mathcal{E}^\beta\|J_t\|_{L^2}^{1/2}+\frac{1}{2}\|\nabla^3u\|_{L^2},
\end{split}
\end{equation}
which directly yields that
\begin{equation}\label{g3u}
\begin{split}
	\|\nabla^3u\|_{L^2} \leq C\mathcal{E}^\beta+C\mathcal{E}^\beta\|\nabla u_t\|_{L^2}+C\mathcal{E}^\beta\|J_t\|_{L^2}^{1/2}.
\end{split}
\end{equation}

Next, multiplying the mass equation $\eqref{9}_1$ by $(\gamma-\frac{1+\delta}{2})\rho^{\gamma-\frac{1+\delta}{2}-1}$, we find that $\varphi\triangleq\rho^{\gamma-\frac{1+\delta}{2}}$ satisfies
\begin{equation}\label{310}
\varphi_t+u\cdot\nabla\varphi+(\gamma-\frac{1+\delta}{2})\varphi\mathrm{div}u=0.
\end{equation}
Integrating \eqref{310} multiplied by $\varphi^5$   on $\Omega_R$ leads to
\begin{equation}\label{e337}
\frac{d}{dt}\|\varphi\|_{L^6}^6
\leq C\|\mathrm{div} u\|_{L^\infty}\|\nabla\varphi\|_{L^6}^6.
\end{equation}
Operating  $\nabla$ to \eqref{310}, one obtains after multiplying the resulting equality by $|\nabla\varphi|^{r-2}\nabla\varphi$ with $r=2,6$  and integrating by parts on $\Omega_R$ that
\begin{equation}\label{e339}
\begin{split}
	\frac{d}{dt}\|\nabla\varphi\|_{L^r}^r
	\leq& C\int|\nabla\varphi|^r|\nabla u|+|\varphi||\nabla\varphi|^{r-1}|\nabla\mathrm{div}u|dx\\
	\leq&C\bigl(\|\nabla u\|_{L^\infty}\|\nabla\varphi\|_{L^r}^r+\|\varphi\|_{L^\infty}\|\nabla\varphi\|_{L^r}^{r-1}\|\nabla\mathrm{div}u\|_{L^r}\bigr).
\end{split}
\end{equation}
Furthermore, operating  $\nabla^2$ to \eqref{310}, multiplying it by $\nabla^2\varphi$ and integrating by parts on $\Omega_R$, one has
\begin{equation}\label{e340}
\begin{split}
\frac{d}{dt}\|\nabla^2\varphi\|_{L^2}^2
\leq& C\int\bigl(|\nabla^2\varphi|^2|\nabla u|+|\nabla\varphi||\nabla^2\varphi||\nabla^2u|+|\varphi||\nabla^2\varphi||\nabla^2\mathrm{div}u|\bigr)dx\\
\leq&C\bigl(\|\nabla u\|_{L^\infty}\|\nabla^2\varphi\|_{L^2}^2+\|\nabla\varphi\|_{L^6}\|\nabla^2\varphi\|_{L^2}\|\nabla^2u\|_{L^3}\\
&+\|\varphi\|_{L^\infty}\|\nabla^2\varphi\|_{L^2}\|\nabla^3u\|_{L^2}\bigr).
\end{split}
\end{equation}

Finally, we deduce from \eqref{e337}-\eqref{e340}, \eqref{g3u}, and Gagliardo-Nirenberg inequality that
\begin{equation}\label{260129}
\frac{d}{dt}\|\varphi\|_{W^{1,6}\cap D^1\cap D^2}\leq C\|\nabla u\|_{H^2}\|\varphi\|_{W^{1,6}\cap D^1\cap D^2}
\leq C\mathcal{E}^\beta+C\|\nabla u_t\|_{L^2}^2+C\|J_t\|_{L^2}^2.
\end{equation}
Thus, \eqref{e336} is derived from integrating \eqref{260129} over $(0,t)$ and using \eqref{e313}.
The proof of Lemma \ref{L3.5} is finished.
\end{proof}

Now, we are in a position to derive the crucial estimates on $\rho^\frac{\delta-1}{2}$.
\begin{lemma}\label{L3.6}
	Let $(\rho,u,H,J)$ and $T_1$ be as in Lemma \ref{L3.1}. Then for all $t\in (0,T_1]$,
	\begin{equation}\label{e343}
		\sup_{0\leq s\leq t}\bigr(\|\rho^\frac{\delta-1}{2}\|_{L^6(\Omega_0)}
        +\|\nabla\rho^\frac{\delta-1}{2}\|_{L^{6}\cap D^{1,2}}\bigr)
		\leq CC_0^\beta+C\int_0^t\mathcal{E}^\beta ds.
	\end{equation}
\end{lemma}
\begin{proof}
First, one can deduce from  the boundary condition $\mathrm{curl}u\times n=0$ on $\partial B_R$ that
\begin{align}\label{curlu}
\mathrm{curlcurl}u\cdot n=0\,\ \text{on}\ \ \partial B_R.
\end{align}
Indeed,  for all $\vartheta\in C^\infty(\overline{\Omega_R})$ with $\vartheta=\nabla \vartheta=0$ on $\partial\Omega$, we have
\begin{equation}
\begin{split}\nonumber
    \int_{\partial B_R} \mathrm{curlcurl} u\cdot n \vartheta dS&=\int_{\partial \Omega_R} \mathrm{curlcurl} u\cdot n \vartheta dS
    =\int_{\Omega_R} \nabla\vartheta\cdot \mathrm{curlcurl} u dx\\
    &=-\int_{\Omega_R}\mathrm{div}\bigl( \nabla\vartheta\times \mathrm{curl} u\bigr) dx=-\int_{\partial \Omega_R}\bigl( \nabla\vartheta\times \mathrm{curl} u\bigr)\cdot n dS\\
    &=-\int_{\partial \Omega_R}\bigl( \mathrm{curl} u\times  n \bigr)\cdot \nabla\vartheta dS=\int_{\partial B_R} \bigl( \mathrm{curl} u\times  n \bigr)\cdot \nabla\vartheta dS=0,
\end{split}
\end{equation}
which thus yields directly  \eqref{curlu}.

Next, we obtain by direct calculations
\begin{equation}\label{3.32}
\begin{split}
&\|\rho^\frac{\delta-1}{2}\mathcal{L}u\|_{L^2}^2\\
&=(\lambda+2\mu)^2\|\rho^\frac{\delta-1}{2}\nabla \mathrm{div}u\|_{L^2}^2+\mu^2\|\rho^\frac{\delta-1}{2}\mathrm{curlcurl}u\|_{L^2}^2\\
&\quad-2\mu(\lambda+2\mu)\int_{\Omega_R}\rho^{\delta-1}\nabla\mathrm{div}u\cdot\mathrm{curlcurl}udx\\
&=(\lambda+2\mu)^2\|\rho^\frac{\delta-1}{2}\nabla \mathrm{div}u\|_{L^2}^2+\mu^2\|\rho^\frac{\delta-1}{2}\mathrm{curlcurl}u\|_{L^2}^2\\
&\quad-2\mu(\lambda+2\mu)\left(\int_{\partial\Omega_R}\rho^{\delta-1}\mathrm{div}u\mathrm{curlcurl}u\cdot ndS-\int_{\Omega_R}\nabla\rho^{\delta-1}\cdot\mathrm{curlcurl}u\mathrm{div}udx\right)\\
&\geq(\lambda+2\mu)^2\|\rho^\frac{\delta-1}{2}\nabla\mathrm{div}u\|_{L^2}^2+\mu^2/2\|\rho^\frac{\delta-1}{2}\mathrm{curlcurl}u\|_{L^2}^2
-C\mathcal{E}^\beta\||\nabla u||\nabla^2 u|\|_{W^{1,1}(\Omega_0)}\\
&\quad-C\|\nabla\rho^\frac{\delta-1}{2}\|_{L^{6}}^2\|\mathrm{div}u\|_{L^{3}}^2,
\end{split}
\end{equation}
where we have used \eqref{qaqa3}, \eqref{rho_lb}, \eqref{g3u}, and \eqref{curlu}.
Hence,  $w\triangleq\rho^\frac{\delta-1}{2}\nabla\mathrm{div}u$ and $\rho^\frac{\delta-1}{2}\mathrm{curlcurl}u$ satisfies the following estimates
\begin{equation}\label{114}
\|w\|_{L^2}+\|\rho^\frac{\delta-1}{2}\mathrm{curlcurl}u\|_{L^2}
\leq C\|\rho^\frac{\delta-1}{2}\mathcal{L}u\|_{L^2}+ C\mathcal{E}^\beta+C\|\nabla u_t\|_{L^2}+C\|J_t\|_{L^2}.
\end{equation}

Multiplying $\eqref{9}_2$ by $\rho^\frac{\delta-1}{2}$ gives that
\begin{equation}
	\begin{split}\label{zsa1}
\rho^\frac{\delta-1}{2}\mathcal{L} u=&\rho^\frac{1-\delta}{2}u_t+\rho^\frac{1-\delta}{2}u\cdot\nabla u+\frac{2\gamma a}{2\gamma-1-\delta}\nabla\rho^{\gamma-\frac{1+\delta}{2}}+\frac{2}{1-\delta}\nabla\rho^\frac{\delta-1}{2}\cdot\mathcal{S}(u)\\
&-J\cdot\nabla H+\nabla H\cdot J,
    \end{split}
\end{equation}
and thus
\begin{equation}
\begin{split}\label{zsa2}
\mu\rho^\frac{\delta-1}{2}\curl\curl u=&-\rho^\frac{1-\delta}{2}u_t-\rho^\frac{1-\delta}{2}u\cdot\nabla u-\frac{2\gamma a}{2\gamma-1-\delta}\nabla\rho^{\gamma-\frac{1+\delta}{2}}-\frac{2}{1-\delta}\nabla\rho^\frac{\delta-1}{2}\cdot\mathcal{S}(u)\\
&+(2\mu+\lambda)\nabla\big(\rho^\frac{\delta-1}{2}\div u\big)-(2\mu+\lambda)\nabla\rho^\frac{\delta-1}{2}\div u+J\cdot\nabla H-\nabla H\cdot J.
\end{split}
\end{equation}
It follows from Lemma \ref{G-N}, \eqref{zsa1}-\eqref{zsa2},  \eqref{rho_sup}, \eqref{qaqa3}, \eqref{g3u}, and \eqref{114} that
\begin{equation}\label{e349}
\begin{split}
&\|\nabla(\rho^\frac{\delta-1}{2}\mathcal{L} u)\|_{L^2}+\|\curl\left(\rho^\frac{\delta-1}{2}\curl\curl u\right)\|_{L^2}\\
&\leq C\Big(\|\rho^\frac{1-\delta}{2}\|_{L^\infty}\|\nabla\rho^\frac{\delta-1}{2}\|_{L^{6}}\|\rho^\frac{1-\delta}{2}u_t\|_{L^3}
+\|\rho^\frac{1-\delta}{2}\|_{L^\infty}\|\mathrm{div}u_t\|_{L^{2}}\\
&\quad+\|\rho^{1-\delta}\|_{L^\infty}\|\nabla\rho^\frac{\delta-1}{2}\|_{L^{6}}\|u\|_{L^6}\|\nabla u\|_{L^6}+\|\rho^\frac{1-\delta}{2}\|_{L^\infty}\|\nabla u\|_{L^{3}}\|\nabla u\|_{L^6}\\
&\quad+\|\rho^{1-\delta}\|_{L^\infty}\|u\|_{L^\infty}\|\nabla^2u\|_{L^{2}}
+\|\rho^{\gamma-\delta}\|_{L^\infty}\|\nabla^2\rho^\frac{\delta-1}{2}\|_{L^{2}}\\
&\quad+\|\rho^\frac{1-\delta}{2}\|_{L^\infty}\|\nabla\rho^{\gamma-\frac{1+\delta}{2}}\|_{L^3}\|\nabla\rho^\frac{\delta-1}{2}\|_{L^{6}}
+\|\nabla^2\rho^\frac{\delta-1}{2}\|_{L^{2}}\|\nabla u\|_{L^\infty}\\
&\quad+\|\rho^\frac{1-\delta}{2}\|_{L^\infty}\|\nabla\rho^\frac{\delta-1}{2}\|_{L^{6}}^2
\|\mathcal{S} (u)\|_{L^{6}}
+\|\nabla\rho^\frac{\delta-1}{2}\|_{L^{6}}\|\nabla^2u\|_{L^3}+\|\nabla J\|_{L^3}\|\nabla H\|_{L^6}\\
&\quad+\|J\|_{L^\infty}\|\nabla^2 H\|_{L^2}\Big)\\
&\leq C\Big(\|\rho\|_{L^\infty}^{3(1-\delta)/4}\|\nabla\rho^\frac{\delta-1}{2}\|_{L^{6}}
\|\rho^\frac{1-\delta}{2}u_t\|_{L^2}^{1/2}\|\nabla u_t\|_{L^2}^{1/2}
+\|\rho\|_{L^\infty}^{(1-\delta)/2}\|\nabla u_t\|_{L^2}\\
&\quad+\|\rho\|_{L^\infty}^{1-\delta}\|\nabla^2\rho^\frac{\delta-1}{2}\|_{L^{2}}\|\nabla u\|_{H^1}^2
+\|\rho\|_{L^\infty}^{1-\delta}\|\nabla u\|_{H^1}\|\nabla^2u\|_{L^2}
+\|\rho\|_{L^\infty}^{\gamma-\delta}\|\nabla^2\rho^\frac{\delta-1}{2}\|_{L^{2}}\\
&\quad+\|\rho\|_{L^\infty}^{(1-\delta)/2}\|\nabla\rho^{\gamma-\frac{1+\delta}{2}}\|_{H^1}\|\nabla\rho^\frac{\delta-1}{2}\|_{L^{6}}
+\|\rho\|_{L^\infty}^{(1-\delta)/2}\|\nabla\rho^\frac{\delta-1}{2}\|_{L^{6}}^2\|\nabla u\|_{H^1}\\
&\quad+\|\nabla\rho^\frac{\delta-1}{2}\|_{L^{6}\cap D^{1}}\|\nabla u\|_{H^1}^{1/2}\|\nabla u\|_{H^2}^{1/2}
+\|\nabla J\|_{L^2}^{1/2}\|\nabla J\|_{H^1}^{1/2}\|\nabla H\|_{H^1}\Big)\\
&\leq C\mathcal{E}^\beta\bigl(1+\|\nabla u_t\|_{L^2}\bigl)+C\|J_t\|_{L^2}.
\end{split}
\end{equation}

Therefore, we can obtain that
\begin{align}\label{98y}
\|w\|_{D^{1}}^2&=\frac{1}{(2\mu+\lambda)^2}\|\rho^\frac{\delta-1}{2}(\mathcal{L}u+\mu\curl\curl u)\|_{D^{1}}^2\nonumber\\
&\leq C\|\nabla(\rho^\frac{\delta-1}{2}\mathcal{L} u)\|_{L^2}^2+C\|\nabla\left(\rho^\frac{\delta-1}{2}\curl\curl u\right)\|_{L^2}^2\nonumber\\
&\leq C\|\nabla(\rho^\frac{\delta-1}{2}\mathcal{L} u)\|_{L^2}^2+C\|\div\left(\rho^\frac{\delta-1}{2}\curl\curl u\right)\|_{L^2}^2+C\|\curl\left(\rho^\frac{\delta-1}{2}\curl\curl u\right)\|_{L^2}^2\nonumber\\
&\quad+C\|\rho^\frac{\delta-1}{2}\curl\curl u\cdot n\|_{H^{1/2}(\partial\Omega)}^2+C\|\rho^\frac{\delta-1}{2}\curl\curl u\|_{L^2}^2\\
&\leq C\|\nabla(\rho^\frac{\delta-1}{2}\mathcal{L} u)\|_{L^2}^2+C\|\nabla\rho^\frac{\delta-1}{2}\|_{L^6}^2\|\nabla^2 u\|_{L^3}^2+C\|\curl\left(\rho^\frac{\delta-1}{2}\curl\curl u\right)\|_{L^2}^2\nonumber\\
&\quad+C\mathcal{E}^\beta\|\curl\curl u\|_{H^1}^2+C\|\rho^\frac{\delta-1}{2}\curl\curl u\|_{L^2}^2\nonumber\\
&\leq C\mathcal{E}^\beta\bigl(1+\|\nabla u_t\|_{L^2}^2\bigr)+C\|\rho^\frac{\delta-1}{2}\mathcal{L}u\|_{L^2}^2+C\|J_t\|_{L^2}^2.\nonumber
\end{align}
owing to Lemma \ref{l22}, \eqref{rho_lb}, \eqref{curlu}, and \eqref{114}.
Additionally, Lemma \ref{G-N} combined with \eqref{114} and \eqref{98y} implies
\begin{align}\label{117}
\|w\|_{L^6}\leq C\|w\|_{H^1}\leq C\|\rho^\frac{\delta-1}{2}\mathcal{L}u\|_{L^2}+C\|J_t\|_{L^2}^2+ C\mathcal{E}^\beta\bigl(1+\|\nabla u_t\|_{L^2}\bigr).
\end{align}

Defined $\varPhi\triangleq\rho^\frac{\delta-1}{2}$, one gets from $\eqref{9}_1$ multiplied by $(\frac{\delta-1}{2})\rho^{\frac{\delta-1}{2}-1}$ that
\begin{equation}\label{38}
\varPhi_t+u\cdot\nabla\varPhi+\frac{\delta-1}{2}\varPhi\mathrm{div}u=0.
\end{equation}
 Integrating \eqref{38} multiplied  by $|\varPhi|^{4}\varPhi$ over $\Omega_0$,  one gets
\begin{equation}
	\begin{split}\label{popza}
		\frac{d}{dt}\|\varPhi\|_{L^6(\Omega_0)}^6
		\leq& C\int_{\Omega_0}\bigl(|u||\nabla\varPhi||\varPhi|^5
		+|\varPhi|^6|\mathrm{div}u|\bigr)dx\\
		\leq&C\bigl(\|u\|_{L^\infty}\|\nabla\varPhi\|_{L^6}\|\varPhi\|_{L^6(\Omega_0)}^5
		+\|\varPhi\|_{L^6(\Omega_0)}^{6}\|\nabla u\|_{L^\infty}\bigr).
	\end{split}
\end{equation}
Now, operating  $\nabla$ to \eqref{38}   and multiplying the resulting equality by $|\nabla\varPhi|^{4}\nabla\varPhi$,  one obtains after integration by parts over $\Omega_R$ that
\begin{equation}
	\begin{split}\label{pop}
		\frac{d}{dt}\|\nabla\varPhi\|_{L^6}^6
		\leq& C\int\bigl(|\nabla\varPhi|^6|\nabla u|
		+|\nabla\varPhi|^5|\rho^\frac{\delta-1}{2}\nabla\mathrm{div}u|\bigr)dx\\
		\leq&C\bigl(\|\nabla u\|_{L^\infty}\|\nabla\varPhi\|_{L^6}^6
		+\|\nabla\varPhi\|_{L^6}^{5}\|\rho^\frac{\delta-1}{2}\nabla\mathrm{div}u\|_{L^6}\bigr).
	\end{split}
\end{equation}
Then, operating  $\nabla^2$ to \eqref{38}  and multiplying the resulting equality by $\nabla^2\varPhi$, integrating it over $\Omega_R$, it holds that
\begin{equation}\label{pop1}
\begin{split}
\frac{d}{dt}\|\nabla^2\varPhi\|_{L^{2}}^2
\leq& C\int\bigl(|\nabla u||\nabla^{2}\varPhi|^2+|\nabla^2u||\nabla\varPhi||\nabla^2\varPhi|+|\nabla^2\varPhi||\nabla(\rho^\frac{\delta-1}{2}\nabla\mathrm{div}u)|\bigr)dx\\
\leq&C\bigl(\|\nabla u\|_{L^\infty}\|\nabla^2\varPhi\|_{L^{2}}^{2}+\|\nabla^2u\|_{L^3}\|\nabla\varPhi\|_{L^{6}}\|\nabla^2\varPhi\|_{L^{2}}\\
&\quad\quad +\|\nabla^2\varPhi\|_{L^{2}}\|\nabla(\rho^\frac{\delta-1}{2}\nabla\mathrm{div}u)\|_{L^{2}}\bigr).
\end{split}
\end{equation}

Finally, we deduce from \eqref{popza}-\eqref{pop1} and \eqref{g3u}, \eqref{98y}, \eqref{117} that
\begin{equation}\nonumber
\begin{split}
&\frac{d}{dt}\left(\|\phi\|_{L^6(\Omega_0)}+\|\nabla\phi\|_{L^6\cap D^{1}}\right)\\
&\leq C\|\nabla u\|_{H^2}\left(\|\phi\|_{L^6(\Omega_0)}+\|\nabla\phi\|_{L^6\cap D^{1}}\right)+C\|\rho^\frac{\delta-1}{2}\nabla\mathrm{div}u\|_{L^6\cap D^{1}}\\
&\leq C\mathcal{E}^\beta+C\|\nabla u_t\|_{L^2}^2+C\|\rho^\frac{\delta-1}{2}\mathcal{L}u\|_{L^2}^2+C\|J_t\|_{L^2}^2.
\end{split}
\end{equation}
Integrating the above inequality over $(0,t)$, using \eqref{e313} and \eqref{e310}, we thus obtain \eqref{e343} and complete the proof of Lemma \ref{L3.6}.
\end{proof}

Now, Proposition \ref{p1} is a direct consequence of Lemmas \ref{L3.1}-\ref{L3.6}.

\begin{solution}
First, direct calculations  imply that
\begin{equation}
    \begin{split}\label{dsxc}
\sup_{0\leq s\leq t}\|\nabla u\|_{L^2}^2&\leq C\sup_{0\leq s\leq t}\big(\|\curl u\|_{L^2}^2+\|\div u\|_{L^2}^2+\|u\|_{L^2(\Omega_0)}^2\big)\\
&\leq C\exp\bigg(C\int_0^t\mathcal{E}^\beta ds\bigg)+C\sup_{0\leq s\leq t}\big(\|\rho^\frac{\delta-1}{2}\|_{L^\infty(\Omega_0)}^2\|\rho^\frac{1-\delta}{2}u\|_{L^2(\Omega_0)}^2\big)\\
&\leq C\exp\bigg(C\int_0^t\mathcal{E}^\beta ds\bigg),
    \end{split}
\end{equation}
where we have used Lemma \ref{l22}, \eqref{e34}, \eqref{e310}, and \eqref{e343}.

Therefore, it follows from \eqref{e34}, \eqref{e310}, \eqref{e313}, \eqref{e336}, and \eqref{e343} that for all $t\in(0,T_1]$,
\begin{equation}
\mathcal{E}^2(t)\leq  CC_0^\beta+C\int_0^t\mathcal{E}^\beta ds+C\exp\bigg(C\int_0^t\mathcal{E}^\beta ds\bigg)\leq C\exp\bigg(C\int_0^t\mathcal{E}^\beta ds\bigg).
\end{equation}
Hence, the standard arguments yield that for $M\triangleq Ce^C$ and $T_0\triangleq \min\{T_1,(CM^\beta)^{-1}\}$,
\begin{equation}\nonumber
	\sup_{0\leq t\leq T_0}\mathcal{E}(t)\leq M.
\end{equation}
which together with \eqref{e34}, \eqref{e310}, and \eqref{e313} gives \eqref{e32}. The proof of Proposition \ref{p1} is completed.
\end{solution}
\begin{remark}
We can also deduce from Lemma \ref{G-N}, \eqref{e310}, \eqref{e313}, \eqref{g3u}, and \eqref{114} that for all $t\in(0,T_0]$,
\begin{equation}\label{3.49}
\begin{split}
&\sup\limits_{0\leq s\leq t}(\|u\|_{L^6}+\|\nabla^2 u\|_{L^2}+\|\nabla^2 H\|_{L^2})
+\int_0^{t}\Big(\|u_t\|_{L^6}^2+\|H_t\|_{L^6}^2+\|\nabla^2 J\|_{L^2}^2\\
&\quad\quad\quad\quad\quad\quad\quad\quad\quad\quad\quad+\|\nabla^3 u\|_{L^2}^2+\|\nabla^3 H\|_{L^2}^2+\|\rho^\frac{\delta-1}{2}\nabla\mathrm{div}u\|_{ L^6\cap D^{1}}^2\Big)ds\leq C.
\end{split}
\end{equation}
\end{remark}

\section{Proof of Theorem 1.1}
\subsection{Existence of the strong solution}\label{4.1}
Letting $(\rho_0, u_0, H_0, J_0)$ be as in Theorem 1.1, since the vacuum only appears in the far-field, we first construct the  approximate smooth $\rho^R_0\in C^\infty(\Omega)$ satisfying that $\rho^R_0>0$ and
\begin{equation}
\left\{ \begin{array}{l}
(\rho^R_0)^\frac{\delta-1}{2}\rightarrow \rho_0^\frac{\delta-1}{2} \in L^6(\Omega_0),\\
(\rho^R_0)^{\gamma-\frac{1+\delta}{2}}\rightarrow \rho_0^{\gamma-\frac{1+\delta}{2}}\ in\ D_0^1(\Omega)\cap D^2(\Omega),\\
\nabla(\rho^R_0)^\frac{\delta-1}{2}\rightarrow \nabla\rho_0^\frac{\delta-1}{2}\ in\  D_0^{1}(\Omega),
       \end{array} \right. ~~~~~\mbox{as}~~~~R\rightarrow\infty.
\end{equation}
Next, we consider the unique strong solution of the following Lam\'{e} system
\begin{equation}
\left\{ \begin{array}{l}
\mathcal{L}u_0^R=-(\rho^R_0)^{1-\delta}u_0^R+(\rho^R_0)^\frac{1-\delta}{2}(\rho_0^\frac{1-\delta}{2}u_0+g)\ast j_\frac{1}{R},\\
u_0^R\cdot n=0,\ \mathrm{curl}u_0^R\times n=-Au,\ x\in\partial \Omega,\\
u_0^R\cdot n=0,\ \mathrm{curl}u_0^R\times n=0,\ x\in\partial B_R.
       \end{array} \right.
\end{equation}
Following the similar arguments as \cite{13}, it can be proved that
\begin{equation}
\lim_{R\rightarrow\infty}(\|u_0^R-u_0\|_{D^1\cap D^2}+\|(\rho^R_0)^\frac{1-\delta}{2}u_0^R-\rho_0^\frac{1-\delta}{2}u_0\|_{L^2})=0.
\end{equation}
Then, we consider the unique strong solution of the linear Stokes system
\begin{equation}
\left\{ \begin{array}{l}
-\Delta H_0^R+H_0^R+\nabla P^R=-\Delta \left(H_0\ast j_\frac{1}{R}\right)+H_0\ast j_\frac{1}{R},\\
H_0^R\cdot n=0,\  \mathrm{curl}H_0^R\times n=0,\ x\in\partial \Omega,\\
H_0^R=0,\ x\in\partial B_R,\\
 \mathrm{div}H_0^R=0.
       \end{array} \right.
\end{equation}
Let $J_0^R=H_0^R(\rho_0^R)^{-\frac{1+\delta}{2}}$, it can be proved that
\begin{equation}
\lim_{R\rightarrow\infty}\big(\|H_0^R-H_0\|_{H^2}+\|J_0^R-J_0\|_{H^1}\big)=0.
\end{equation}

According to Lemma \ref{l25}, the IBVP \eqref{9} with the initial data $(\rho^R_0,u_0^R,H_0^R,J_0^R)$ has a classical solution $(\rho^R,u^R,H^R,J^R)$ on $\Omega_R\times[0,T_R]$. Moreover, Proposition \ref{p1} shows that  there exists a $T_0$ independent of $R$ such that both \eqref{e32} and \eqref{3.49} hold for $(\rho^R,u^R, H^R, J^R)$. Extending $(\rho^R,u^R,\phi^R:=\nabla(\rho^R)^\frac{\delta-1}{2},H^R,J^R)$ by zero on $\Omega\setminus \Omega_R$ and denoting
\begin{equation}
\tilde{\rho}^R=\rho^R\varphi_R^\nu,
~~\tilde{\phi}^R=\nabla(\rho^R)^\frac{\delta-1}{2}\varphi_R,~~w^R=u^R\varphi_R,~~\tilde{H}^R=H^R\varphi_R,~~\tilde{J}^R=J^R\varphi_R,
\end{equation}
with $\nu=\frac{2}{2\gamma-1-\delta}$ and $\varphi_R\in C_0^\infty(B_R)$ satisfying
\begin{equation}\label{phi_R}
    0\leq\varphi_R\leq 1,~~\varphi_R(x)\leq 1,~\text{for}~\vert x\vert\leq R/2,~~\vert\nabla^k\varphi_R\vert\leq CR^{-k}~(k=1,2,3),
\end{equation}
for $R> 2R_0+1$, we thus deduce from Proposition \ref{p1}, \eqref{3.49}, and \eqref{phi_R} that
\begin{equation}\label{4.6}
\begin{split}
&\sup_{0\leq t\leq T_0}(\|(\tilde{\rho}^R)^\frac{1-\delta}{2}w^R\|_{L^2(\Omega)}+\|\nabla w^R\|_{L^2(\Omega)}
+\|(\tilde{\rho}^R)^\frac{1-\delta}{2}w^R_t\|_{L^2(\Omega)})\\
&\leq C+C\sup_{0\leq t\leq T_0}(\|\nabla u^R\|_{L^2(\Omega_R)}+\|u^R\|_{L^6(\Omega_R)}\|\nabla \varphi_R\|_{L^3(\Omega_R)})\\
&\leq C,
\end{split}
\end{equation}
\begin{equation}
\begin{split}
&\sup_{0\leq t\leq T_0}(\|w^R\|_{L^6(\Omega)}+\|\nabla^2w^R\|_{L^2(\Omega)})\\
&\leq C+C\sup_{0\leq t\leq T_0}(\|\nabla^2 u^R\|_{L^2(\Omega_R)}+\frac{1}{R}\|\nabla u^R\|_{L^2(\Omega_R)}
+\|u^R\|_{L^6(\Omega_R)}\|\nabla^2 \varphi_R\|_{L^3(\Omega_R)})\\
&\leq  C,
\end{split}
\end{equation}
\begin{equation}
\begin{split}
&\sup_{0\leq t\leq T_0}(\|\tilde H^R\|_{L^2(\Omega)}+\|\nabla \tilde H^R\|_{L^2(\Omega)}
+\|\tilde H^R_t\|_{L^2(\Omega)})\\
&\leq C+C\sup_{0\leq t\leq T_0}(\|\nabla H^R\|_{L^2(\Omega_R)}+\| H^R\|_{L^6(\Omega_R)}\|\nabla \varphi_R\|_{L^3(\Omega_R)})\\
&\leq  C,
\end{split}
\end{equation}
\begin{equation}
\begin{split}
&\sup_{0\leq t\leq T_0}(\|\tilde H^R\|_{L^6(\Omega)}+\|\nabla^2\tilde H^R\|_{L^2(\Omega)})\\
&\leq C+C\sup_{0\leq t\leq T_0}(\|\nabla^2 H^R\|_{L^2(\Omega_R)}+\frac{1}{R}\|\nabla H^R\|_{L^2(\Omega_R)}
+\|H^R\|_{L^6(\Omega_R)}\|\nabla^2 \varphi_R\|_{L^3(\Omega_R)})\\
&\leq  C,
\end{split}
\end{equation}
and
\begin{equation}
\begin{split}
&\sup_{0\leq t\leq T_0}(\|\tilde J^R\|_{L^2(\Omega)}+\|\nabla \tilde J^R\|_{L^2(\Omega)})\\
&\leq C+C\sup_{0\leq t\leq T_0}(\|\nabla J^R\|_{L^2(\Omega_R)}+\| J^R\|_{L^6(\Omega_R)}\|\nabla \varphi_R\|_{L^3(\Omega_R)})\\
&\leq  C.
\end{split}
\end{equation}

Next, the straightforward calculations yield that
\begin{equation}
\begin{split}
&\sup_{0\leq t\leq T_0}\|(\tilde{\rho}^R)^{\gamma-\frac{1+\delta}{2}}\|_{W^{1,6}\cap D^1\cap D^2(\Omega)}\\
&\leq C+C\sup_{0\leq t\leq T_0}(\|\nabla (\rho^R)^{\gamma-\frac{1+\delta}{2}}\|_{L^2\cap L^6(\Omega_R)}
+\|(\rho^R)^{\gamma-\frac{1+\delta}{2}}\|_{W^{1,6}(\Omega_R)}\|\nabla \varphi_R\|_{L^3\cap L^6(\Omega_R)})\\
&+C\sup_{0\leq t\leq T_0}(\|\nabla^2(\rho^R)^{\gamma-\frac{1+\delta}{2}}\|_{L^2(\Omega_R)}
+\|(\rho^R)^{\gamma-\frac{1+\delta}{2}}\|_{L^6(\Omega_R)}\|\nabla \varphi_R\|_{L^3(\Omega_R)})\\
&\leq C,
\end{split}
\end{equation}
and that
\begin{equation}
\begin{split}
&\sup_{0\leq t\leq T_0}\|\tilde{\phi}^R\|_{L^{6}\cap D^{1,2}(\Omega)}\\
&\leq C+C\sup_{0\leq t\leq T_0}(\|\nabla^2 (\rho^R)^\frac{\delta-1}{2}\|_{L^{2}(\Omega_R)}
+\|\nabla (\rho^R)^\frac{\delta-1}{2}\|_{L^{6}(\Omega_R)}\|\nabla \varphi_R\|_{L^3(\Omega_R)})\\
&\leq C.
\end{split}
\end{equation}

Similarly, it follows from Proposition \ref{p1} and \eqref{3.49} that
\begin{equation}
    \int_0^{T_0}\bigl(\|\nabla w^R_t\|_{L^2(\Omega)}+\|\nabla \tilde H^R_t\|_{L^2(\Omega)}+\|\tilde J^R_t\|_{L^2(\Omega)}+\|\nabla^2 \tilde J^R\|_{L^2(\Omega)}^2\bigr)dt
  \leq C,
\end{equation}
\begin{equation}
    \int_0^{T_0}\bigl(\| w^R_t\|_{L^6(\Omega)}^2+\|\nabla^3 w^R\|_{L^2(\Omega)}^2+\| \tilde H^R_t\|_{L^6(\Omega)}^2+\|\nabla^3 \tilde H^R\|_{L^2(\Omega)}^2\bigr)dt
  \leq C,
\end{equation}
and 
\begin{equation}\label{4.12}
\begin{split}
\int_0^{T_0}\|\bigl((\tilde{\rho}^R)^{\gamma-\frac{1+\delta}{2}}\bigr)_t\|_{L^6(\Omega)}^2dt\leq
    & \int_0^{T_0}\|u\cdot\nabla(\rho^R)^{\gamma-\frac{1+\delta}{2}}\|_{L^6(\Omega_R)}^2
    +\|(\rho^R)^{\gamma-\frac{1+\delta}{2}}\mathrm{div}u\|_{L^6(\Omega_R)}^2dt\\
    \leq&\int_0^{T_0}\|(\rho^R)^{\gamma-\frac{1+\delta}{2}}\|_{W^{1,6}(\Omega_R)}^2
    \|\nabla u\|_{H^1(\Omega_R)}^2dt\\
  \leq& C,
\end{split}
\end{equation}
due to \eqref{310}.
With all these estimates \eqref{4.6}-\eqref{4.12} at hand, there exists a subsequence $R_j$, $R_j\rightarrow\infty$, such that $(\tilde{\rho}^{R_j},w^{R_j},\tilde{\phi}^{R_j},\tilde{H}^{R_j},\tilde{J}^{R_j})$ converge to some limit $(\rho,u,\phi,H,J)$ in some weak sense as follows:
\begin{align}
	\rho^{R_j}\rightarrow\rho,u^{R_j}\rightarrow u,H^{R_j}\rightarrow H\ in\ C(\overline{\Omega_N}\times[0,T_0]),\forall\ N>2R_0+1,\label{strong}\\
	(\rho^{R_j})^{\gamma-\frac{1+\delta}{2}}\rightharpoonup \rho^{\gamma-\frac{1+\delta}{2}}\ \mathrm{weakly}^\ast\ in\ L^\infty(0,T_0; D_0^1\cap D^2(\Omega)),\\
	\tilde{\phi}^R\rightharpoonup\phi\ \mathrm{weakly}^\ast\ L^\infty(0,T_0;D_0^{1}(\Omega)),\\
	u^{R_j}\rightharpoonup u,H^{R_j}\rightharpoonup H\ \mathrm{weakly}^\ast\ in\ L^\infty(0,T_0;D_0^1(\Omega)\cap D^2(\Omega)),\\
	J^{R_j}\rightharpoonup J\ \mathrm{weakly}^\ast\ in\ L^\infty(0,T_0;L^2(\Omega)\cap D^1(\Omega)),\\
    \nabla^3u^{R_j}\rightharpoonup\nabla^3u,\nabla^3H^{R_j}\rightharpoonup\nabla^3H\ \mathrm{weakly}\ in\ L^2((0,T_0)\times\Omega),\\
	u^{R_j}_t\rightharpoonup u_t\ \mathrm{weakly}\ in\ L^2(0,T_0;D_0^1(\Omega)),\\
	H^{R_j}_t\rightharpoonup H_t\ \mathrm{weakly}\ in\ L^2(0,T_0;H^1(\Omega)),\\
	J^{R_j}_t\rightharpoonup J_t\ \mathrm{weakly}\ in\ L^2((0,T_0)\times\Omega).\label{weak}
\end{align}
And, it is easy to check that $\phi=\nabla \rho^\frac{\delta-1}{2}$ in the weak sense, since $\tilde{\phi}^R=\nabla(\tilde{\rho}^R)^\frac{\delta-1}{2}$ in $B_{R/2}$ for all $R>2R_0+1$. Next, for any test function $\Psi\in C_0^\infty(\Omega\times[0,T_0])$, it follows form \eqref{strong}-\eqref{weak} that
\begin{equation}
    \begin{split}
    \lim_{j\rightarrow\infty}\int_0^{T_0}\int_{\Omega}(\rho^{R_j})^\frac{\delta-1}{2}\mathcal{L}u^{R_j}\varphi_{R_j}\Psi dx dt
        &=\lim_{j\rightarrow\infty}\int_0^{T_0}\int_{supp\Psi}(\rho^{R_j})^\frac{\delta-1}{2}\mathcal{L}u^{R_j}\varphi_{R_j}\Psi dx dt\\
        &=\lim\limits_{\substack{j\rightarrow\infty \\supp\Psi\subset B_{R_j/2}}}\int_0^{T_0}\int_{supp\Psi}(\rho^{R_j})^\frac{\delta-1}{2}\mathcal{L}u^{R_j}\varphi_{R_j}\Psi dxdt\\
        &=\lim\limits_{\substack{j\rightarrow\infty \\supp\Psi\subset B_{R_j/2}}}\int_0^{T_0}\int_{\Omega}(\tilde{\rho}^{R_j})^\frac{\delta-1}{2}\mathcal{L}w^{R_j}\Psi dx dt\\
        &=\int_0^{T_0}\int_{\Omega}\rho^\frac{\delta-1}{2}\mathcal{L}u\Psi dx dt,
    \end{split}
\end{equation}
and thus
\begin{align}
        (\rho^{R_j})^\frac{\delta-1}{2}\mathcal{L}u^{R_j}\varphi_{R_j}\rightharpoonup \rho^\frac{\delta-1}{2}\mathcal{L}u\ \mathrm{weakly}\ in\ L^2((0,T_0)\times\Omega).
\end{align}
Similarly, we can obtain 
\begin{align}
	(\rho^{R_j})^\frac{\delta-1}{2}\nabla\mathrm{div}u^{R_j}\varphi_{R_j}\rightharpoonup\rho^\frac{\delta-1}{2}\nabla\mathrm{div}u\ \mathrm{weakly}\ in\ L^2((0,T_0)\times\Omega),\\
	\nabla((\rho^{R_j})^\frac{\delta-1}{2}\nabla\mathrm{div}u^{R_j})\varphi_{R_j}\rightharpoonup\nabla(\rho^\frac{\delta-1}{2}\nabla\mathrm{div}u)\ \mathrm{weakly}\ in\  L^2((0,T_0)\times\Omega).
\end{align}
Moreover, $(\rho,u,H,J)$ also satisfies \eqref{e32} and \eqref{3.49}. Therefore, we can easily show that $(\rho,u, H, J)$ is a strong solution to the system \eqref{q1q} and thus obtain the strong solution $(\rho,u, H,)$ to the IBVP \eqref{1}-\eqref{7}  satisfying the regularity \eqref{e116}.

\subsection{Uniqueness of the strong solution}
It remains to prove the uniqueness of the solution constructed in Section \ref{4.1}. Assume that $(\rho_1,u_1,H_1,J_1)$ and $(\rho_2,u_2,H_2,J_2)$ are two strong solutions satisfying \eqref{e116} with the same initial data. Set
$$
\begin{gathered}
h_i=\rho^\frac{1-\delta}{2}_i,\ \varphi_i =\rho^{\gamma-\delta}_i,\ \psi_i=\nabla\log\rho_i,\ g_i=\rho^{-\delta},\ i=1,2,\\
\bar{h}=h_1-h_2,\ \bar{\varphi}=\varphi_1-\varphi_2,\ \bar{\psi}=\psi_1-\psi_2,\ \bar{g}=g_1-g_2,\
\bar{u}=u_1-u_2,\ \bar H=H_1-H_2.
\end{gathered}
$$
Meanwhile, $(\bar h, \bar{\varphi}, \bar{\psi}, \bar{g}, \bar{u}, \bar{H})$ meets the following initial conditions
\begin{equation*}
	(\bar h, \bar{\varphi}, \bar{\psi}, \bar{g}, \bar{u}, \bar{H})|_{t=0}=(0,0,0,0,0),
\end{equation*}
with boundary conditions
\begin{equation*}
	\overline{u}\cdot n|_{\partial\Omega}=0,\quad \mbox{curl}\overline{u}\times n|_{\partial\Omega}=-A\overline{u},
\end{equation*}
\begin{equation*}
	\overline{H}\cdot n|_{\partial\Omega}=0, \quad \mbox{curl}\overline{H}\times n|_{\partial\Omega}=0.
\end{equation*}

Next, our proof is mainly divided into three parts.

\textbf{Step 1 (estimate of $\bar u$)}:
Subtracting the momentum equations satisfied by $(\rho_1, u_1, H_1)$ and $(\rho_2, u_2, H_2)$ yields
\begin{equation}\label{2.96}
\begin{split}
&h_1^2\bar{u}_t+h_1^2u_1\cdot\nabla\bar{u}-\mathcal{L}\bar{u}=-\nabla\bar{\varphi}-\bar{h}(h_1+h_2)(u_2)_t
-\bar{h}(h_1+h_2)u_2\cdot\nabla u_2\\
&\quad-h_1^2\bar{u}\cdot\nabla u_2+\bar{\psi}\cdot\mathcal{S}(u_2)+\psi_1\cdot\mathcal{S}(\bar{u})
+g_1H_1\cdot\nabla \bar H+\bar gH_1\cdot\nabla H_2+g_2\bar H\cdot\nabla H_2\\
&\quad -g_1\nabla \bar H\cdot H_1-\bar g\nabla H_2\cdot H_1-g_2\nabla H_2\cdot \bar H,
\end{split}
\end{equation}
where (and in the following)  we   take all the coefficients as 1 for simplicity. 

First, observing that
\begin{align}\label{lop1}
g_1H_1=h_1J_1,\ g_2\nabla H_2=h_2\psi_2\otimes J_2+h_2\nabla J_2,
\end{align}
one obtains after multiplying \eqref{2.96} by $\bar{u}$ and integrating by parts that 
\begin{equation}
\begin{split}\label{wsw2}
&\frac{1}{2}\frac{d}{dt}\|h_1\bar u\|_{L^2}^2+\mu\|\mathrm{curl} \bar u\|_{L^2}^2+(2\mu+\lambda)\|\mathrm{div}\bar u\|_{L^2}^2+\mu\int_{\partial\Omega}\bar u\cdot A\cdot \bar udS\\
&\leq C\Big(\|\mathrm{div}u_1\|_{L^\infty}\|h_1\bar{u}\|_{L^2}^2+\|\bar{\varphi}\|_{L^2}\|\nabla\bar u\|_{L^2}
+\|\bar{h}\|_{L^2}\|h_1\bar{u}\|_{L^3}\|(u_2)_t\|_{L^6}\\
&\quad+\|\bar{h}\|_{L^2}\|h_2(u_2)_t\|_{L^3}\|\bar{u}\|_{L^6}+\|\bar{h}\|_{L^2}\|h_1+h_2\|_{L^\infty}\|u_2\|_{L^6}\|\nabla u_2\|_{L^6}\|\bar{u}\|_{L^6}\\
&\quad+\|\nabla u_2\|_{L^\infty}\|h_1\bar{u}\|_{L^2}^2+\|\bar{\psi}\|_{L^2}\|\nabla u_2\|_{L^3}\|\bar{u}\|_{L^6}\\
&\quad +\|(h_1)^{-1}\psi_1\|_{L^{6}}\|\nabla\bar{u}\|_{L^2}\|h_1\bar{u}\|_{L^3}+\|J_1\|_{L^6}\|\nabla\bar H\|_{L^2}\|h_1\bar u\|_{L^3}\\
&\quad +\|\bar g\|_{L^2}\|H_1\|_{L^6}\|\nabla H_2\|_{L^6}\|\bar u\|_{L^6}+\|h_2\|_{L^\infty}\|\psi_2\|_{L^6}\|J_2\|_{L^6}\|\bar H\|_{L^2}\|\bar u\|_{L^6}\\
&\quad +\|h_2\|_{L^\infty}\|\nabla J_2\|_{L^2}\|\bar H\|_{L^3}\|\bar u\|_{L^6}\Big)\\
&\leq \epsilon\|\nabla \bar u\|_{L^2}^2+\epsilon\|\nabla \bar H\|_{L^2}^2\\
&\quad+CF(t)\Big(\|h_1\bar{u}\|_{L^2}^2+\|\bar{\varphi}\|_{L^2}^2+\|\bar{h}\|_{L^2}^2+\|\bar\psi\|_{L^2}^2
+\|\bar g\|_{L^2}^2
+\|\bar{H}\|_{L^2}^2\Big),
\end{split}
\end{equation}
where
$$
F(t)=1+\sum_{i=1}^{2}\big(\| u_i\|_{D^3}^2+\|(u_i)_t\|_{D^1}^2+\|H_i\|_{D^3}^2+\| (H_i)_t\|_{D^1}^2+\|J_i\|_{D^2}^2+\| (J_i)_t\|_{L^2}^2\big).
$$
Using the trace theorem, the boundary term in \eqref{za3.4} can be governed as
\begin{align}\label{2trace2}
\int_{\partial\Omega} \bar u\cdot A\cdot\bar udS\leq C\||\bar u|^2\|_{W^{1,1}(\Omega_0)}  &\leq C\|h_1\bar{u}\|_{L^2}^2+\epsilon\|\nabla u\|_{L^2}^2,
\end{align}
where we used
\begin{equation}\label{zrho_inf1}
    \|h_1^{-1}\|_{L^\infty(\Omega_0)}=\|h_1^{-1}\|_{W^{1,6}(\Omega_0)}\leq C.
\end{equation}

Next, it follows from, \eqref{2.96}, and  the standard $L^p$ estimate theory that 
\begin{equation}\label{qaza}
\begin{split}
\|\nabla^2\bar{u}\|_{L^2}\leq C\|\mathcal{L}\bar{u}\|_{L^2}+C\|\nabla\bar{u}\|_{L^2}
\leq &C\|h_1\bar{u}_t\|_{L^2}+F(t)(\|\bar{h}\|_{H^1}+\|\nabla\bar{u}\|_{L^2}+\|\nabla\bar{\varphi}\|_{L^2})\\
&+CF(t)(\|\bar\psi\|_{L^2}+\|\bar{g}\|_{L^2}+\|\bar H\|_{H^1}),
\end{split}
\end{equation}
and that
\begin{align}\label{qaza1}
h_1^2(u_1)_t-h_2^2(u_2)_t=\frac12\Big(h_1^2\bar u_t+(h_1+h_2)\bar h(u_2)_t\Big)+\frac12\Big((h_1+h_2)\bar h(u_1)_t+h_2^2\bar u_t\Big).
\end{align}
Then, multiplying \eqref{2.96} by $2\bar{u}_t$ and integrating by parts lead to
\begin{equation}
\begin{split}\label{wsw2q}
&\frac{d}{dt}\left(\mu\|\mathrm{curl}\bar u\|_{L^2}^2+(2\mu+\lambda)\|\mathrm{div}\bar u\|_{L^2}^2+\mu\int_{\partial\Omega} \bar u\cdot A\cdot\bar udS\right)+\|h_1\bar u_t\|_{L^2}^2+\|h_2\bar u_t\|_{L^2}^2\\
&= 2\int\Big(-\frac12(h_1+h_2)\bar{h}\big((u_2)_t+(u_1)_t\big)-h_1^2u_1\cdot\nabla \bar u-\nabla\bar{\varphi}
\\
&\quad-\bar{h}(h_1+h_2)u_2\cdot\nabla u_2-h_1^2\bar{u}\cdot\nabla u_2+\bar\psi\cdot\mathcal{S}(u_2)+\psi_1\cdot\mathcal{S}(\bar{u})
+g_1H_1\cdot\nabla \bar H\\
&\quad +\bar gH_1\cdot\nabla H_2+g_2\bar H\cdot\nabla H_2-g_1\nabla \bar H\cdot H_1-\bar g\nabla H_2\cdot H_1-g_2\nabla H_2\cdot \bar H\Big)\cdot\bar{u}_tdx\\
&=\sum_{i=1}^{13}K_i,
\end{split}
\end{equation}
We will estimate each $K_i(i=1,...,13)$ as follows:
\begin{equation}
\begin{split}\nonumber
&K_1\leq C\|\bar h\|_{L^3}\big(\|(u_2)_t\|_{L^6}+\|(u_1)_t\|_{L^6}\big)\|(h_1+h_2) \bar u_t\|_{L^2}\\
&\quad\ \leq CF(t)\|\bar{h}\|_{H^1}^2
+\epsilon\|(h_1+h_2)\bar{u}_t\|_{L^2}^2,\\
&K_2\leq C\|h_1\|_{L^\infty}\|u_1\|_{L^\infty}\|\nabla \bar u\|_{L^2}\|h_1\bar u_t\|_{L^2}
\leq CF(t)\|\nabla\bar u\|_{L^2}^2+\epsilon\|(h_1+h_2)\bar{u}_t\|_{L^2}^2,\\
&K_4\leq C\|\bar h\|_{L^6}\|u_2\|_{L^6}\|\nabla u_2\|_{L^6}\|(h_1+h_2)\bar u_t\|_{L^2} \leq C\|\bar{h}\|_{H^1}^2+\epsilon\|(h_1+h_2)\bar{u}_t\|_{L^2}^2,\\
&K_5\leq C\|h_1\|_{L^\infty}\|\bar{u}\|_{L^6}\|\nabla u_2\|_{L^3}\|h_1\bar{u}_t\|_{L^2}\leq C\|\nabla\bar u\|_{L^2}+\epsilon\|h_1\bar{u}_t\|_{L^2}^2,\\
&K_7\leq C\|h_1^{-1}\psi_1\|_{L^{6}}\|\nabla\bar{u}\|_{L^{3}}\|h_1\bar{u}_t\|_{L^2}\leq C\|\nabla\bar u\|_{L^2}^2+\epsilon\|\nabla^2\bar u\|_{L^2}^2+\epsilon\|h_1\bar{u}_t\|_{L^2}^2,\\
&K_8+K_{11}\leq C\|J_1\|_{L^\infty}\|\nabla\bar H\|_{L^2}\|h_1\bar u_t\|_{L^2}\leq CF(t)\|\nabla\bar H\|_{L^2}^2+\epsilon\|h_1\bar{u}_t\|_{L^2}^2,\\
&K_{9}+K_{12}\leq C\|g_1^{-1}\|_{L^\infty}\|J_1\|_{L^\infty}\|\bar g\|_{L^2}\|\nabla H_2\|_{L^\infty}\|h_1\bar u_t\|_{L^2}\leq CF(t)\|\bar g\|_{L^2}^2+\epsilon\|h_1\bar{u}_t\|_{L^2}^2,\\
&K_{10}+K_{13}\leq C\|\bar H\|_{L^3}(\|\psi_2\|_{L^6}\|J_2\|_{L^\infty}+\|\nabla J_2\|_{L^6})\|h_2\bar u_t\|_{L^2}\leq CF(t)\|\bar H\|_{H^1}^2+\epsilon\|h_2\bar{u}_t\|_{L^2}^2,
\end{split}
\end{equation}
\begin{equation}
	\begin{split}\nonumber
		K_3=&2\frac{d}{dt}\int\bar\varphi\div\bar{u}dx-2\int \bar{\varphi}_t\div\bar{u}dx\\ \leq&-2\frac{d}{dt}\int\bar{h}\nabla\varphi_2\cdot\bar{u}dx+\|\nabla u_1\|_{L^3}\|\bar\varphi\|_{L^2}\|\nabla\bar{u}\|_{L^6}\\ &+\|u_1\|_{L^\infty}\|\bar\varphi\|_{L^2}\|\nabla^2\bar{u}\|_{L^2}+\|\bar\varphi\|_{L^2}\|\nabla u_2\|_{L^3}\|\nabla\bar{u}\|_{L^6}\\
		&+\|\varphi_1\|_{L^\infty}\|\nabla\bar{u}\|_{L^2}^2+\|\bar{u}\|_{L^6}\|\nabla\varphi_1\|_{L^3}\|\nabla\bar{u}\|_{L^2},
	\end{split}
\end{equation}
and
\begin{equation}
\begin{split}\nonumber
K_6=&2\frac{d}{dt}\int \bar\psi\cdot\mathcal{S}(u_2)\cdot\bar{u}dx-2\int\bar\psi_t\cdot\mathcal{S}(u_2)\cdot\bar{u}dx-2\int\bar\psi\cdot\mathcal{S}(u_2)_t\cdot\bar{u}dx\\
\leq&2\frac{d}{dt}\int \bar\psi\cdot\mathcal{S}(u_2)\cdot\bar{u}dx+\|\bar\psi\|_{L^2}\|\nabla u_1\|_{L^6}\|\nabla u_2\|_{L^6}\|\bar{u}\|_{L^6}\\
&+\|\bar\psi\|_{L^2}\|u_1\|_{L^\infty}\|\nabla^2 u_2\|_{L^3}\|\bar{u}\|_{L^6}+\|\bar\psi\|_{L^2}\| u_1\|_{L^\infty}\|\nabla u_2\|_{L^6}\|\nabla\bar{u}\|_{L^3}\\
&+\|\bar\psi\|_{L^2}\|\nabla u_1\|_{L^6}\|\nabla u_2\|_{L^6}\|\bar u\|_{L^6}
+\|\nabla \psi_2\|_{L^{2}}\|\bar{u}\|_{L^6}^2\|\nabla u_2\|_{L^{6}}\\
&+\|\psi_2\|_{L^{6}}\|\nabla\bar{u}\|_{L^2}\|\nabla u_2\|_{L^{6}}\|\bar{u}\|_{L^6}
+\|\nabla^2\bar{u}\|_{L^2}\|\nabla u_2\|_{L^3}\|\bar{u}\|_{L^6}\\
&+\|\bar\psi\|_{L^2}\|\nabla(u_2)_t\|_{L^2}\|\bar{u}\|_{L^\infty}\\
\leq&2\frac{d}{dt}\int \bar\psi\cdot\mathcal{S}(u_2)\cdot\bar{u}dx+CF(t)(\|\bar\psi\|_{L^2}^2+\|\nabla\bar u\|_{L^2}^2)+\epsilon\|\nabla^2 \bar u\|_{L^2}^2.
\end{split}
\end{equation}

Substituting $K_1$--$K_{13}$ into \eqref{wsw2q} and choosing $\epsilon$ small enough,  one gets
\begin{equation}
\begin{split}\label{wsa3}
&\frac{d}{dt}\left(\mu\|\mathrm{curl}\bar u\|_{L^2}^2+(2\mu+\lambda)\|\mathrm{div}\bar u\|_{L^2}^2+K(t)\right)+\|h_1\bar u_t\|_{L^2}^2+\|h_2\bar u_t\|_{L^2}^2\\ &\leq CF(t)\big(\|\bar{h}\|_{H^1}^2+\|\nabla\bar{u}\|_{L^2}^2+\|h_1\bar{u}\|_{L^2}^2+\|\bar\varphi\|_{H^1}^2+\|\bar\psi\|_{L^2}^2+\|\bar H\|_{H^1}^2\big),
\end{split}
\end{equation}
where
\begin{equation}
K(t)=\mu\int_{\partial\Omega} \bar u\cdot A\cdot\bar udS-2\int\bar\varphi\div\bar{u}dx-2\int\bar\psi\cdot\mathcal{S}(u_2)\cdot\bar{u}dx 
\end{equation}
satisfying
\begin{equation}
|K(t)|\leq \epsilon\|\nabla\bar{u}\|_{L^2}^2+C(\|h_1\bar u\|_{L^2}^2+\|\bar\varphi\|_{L^2}^2+\|\psi\|_{L^2}^2).
\end{equation}

\textbf{Step 2 (estimate of $\bar H$)}: Subtracting the magnetic field equations satisfied by $(u_1,H_1)$ and $(u_2,H_2)$ yields
\begin{equation}
\begin{split}\label{wsw21}
\bar{H}_{t}- \eta \Delta\bar{ H}=-u_{1}\cdot\nabla\bar{H}-\bar{u}\cdot\nabla H_2
+H_1\cdot\nabla \bar{u}+\bar{H}\cdot\nabla u_2-H_1\mbox{div}\bar{u}-\bar{H}\mbox{div}u_2.
\end{split}
\end{equation}
First, multiplying \eqref{wsw21} by $\bar{H}$ and integrating by parts lead  to
\begin{equation}\label{qaqa4u1}
\begin{split}
&\frac{1}{2}\frac{d}{dt}\|\bar H\|_{L^2}^2+\eta\|\curl \bar H\|^2\\
&\leq C\big(\|u_1\|_{L^\infty}\|\nabla \bar H\|_{L^2}+\|\bar u\|_{L^6}\|\nabla H_2\|_{L^3}+\| H_1\|_{L^\infty}\|\nabla\bar u\|_{L^2}+\|\bar H\|_{L^6}\|\nabla u_2\|_{L^3}\big)\|\bar H\|_{L^2}\\
&\leq \epsilon\|\nabla \bar u\|_{L^2}^2+\epsilon\|\nabla \bar H\|_{L^2}^2+F(t)\|\bar H\|_{L^2}^2.
\end{split}
\end{equation}
Next,   multiplying \eqref{wsw21} by $\bar H_t$ and integrating the resulting equality by parts, we obtain
\begin{equation}
\begin{split}\label{1wsw21}
&\frac{\eta}{2}\frac{d}{dt}\|\mathrm{curl}\bar H\|_{L^2}^2+\|\bar H_t\|_{L^2}^2\\
&\leq C\big(\|u_1\|_{L^\infty}\|\nabla \bar H\|_{L^2}+\|\bar u\|_{L^6}\|\nabla H_2\|_{L^3}+\| H_1\|_{L^\infty}\|\nabla\bar u\|_{L^2}+\|\bar H\|_{L^6}\|\nabla u_2\|_{L^3}\big)\|\bar H_t\|_{L^2}\\
&\leq \epsilon\|\bar H_t\|_{L^2}^2+CF(t)(\|\nabla \bar H\|_{L^2}^2+\|\nabla \bar u\|_{L^2}^2).
\end{split}
\end{equation}

\textbf{Step 3 (estimate of $\bar\rho$)}: Subtracting the mass equations satisfied by $(\rho_1,u_1)$ and $(\rho_2,u_2)$ yields
\begin{gather}
\bar{h}_t+u_2\cdot\nabla\bar{h}+\mathrm{div}u_2\bar{h}=-\bar{u}\cdot\nabla h_1-\mathrm{div}\bar{u}h_1,\label{2.98}\\
\bar{\varphi}_t+u_2\cdot\nabla\bar{\varphi}+\div u_2\bar\varphi=-\bar{u}\cdot\nabla\varphi_1-\div\bar u\varphi_1,\label{2.100}\\
\bar{\psi}_t+u_2\cdot\nabla\bar\psi+ \nabla u_2\cdot\bar{\psi}=-\bar{u}\cdot\nabla\psi_1-\nabla \bar{u}\cdot\psi_1-\nabla\mathrm{div}\bar{u}.\label{2.99}
\end{gather}

First, multiplying \eqref{2.98} by $\bar{h}$ and integrating by parts give that 
\begin{equation}
	\begin{split}\label{cx1}
		\frac{d}{dt}\|\bar{h}\|_{L^2}^2\leq&\|\mathrm{div}u_2\|_{L^\infty}\|\bar{h}\|_{L^2}^2 +\|h_1\bar{u}\|_{L^{3}}\|h_1^{-1}\nabla h_1\|_{L^{6}}\|\bar{h}\|_{L^2} +\|\nabla\bar{u}\|_{L^2}\|h_1\|_{L^\infty}\|\bar{h}\|_{L^2}\\
		\leq& \epsilon\|\nabla \bar{u}\|_{L^2}^2+CF(t)\Big(\|\bar{h}\|_{L^2}^2+\|h_1\bar{u}\|_{L^2}^2\Big).
	\end{split}
\end{equation}
Operating $\nabla$ to \eqref{2.98}, multiplying the resulting equation  by $\nabla\bar{h}$ and integrating by parts, it holds 
\begin{equation}
\begin{split}
\frac{d}{dt}\|\nabla\bar{h}\|_{L^2}^2\leq
&\|\nabla u_2\|_{L^\infty}\|\nabla\bar{h}\|_{L^2}^2+\|\nabla^2 u_2\|_{L^6}\|\bar{h}\|_{L^3}\|\nabla\bar{h}\|_{L^2}
\\
&+\|h_1\|_{L^\infty}^3\|\nabla^2(h_1)^{-1}\|_{L^{2}}\|\bar{u}\|_{L^{\infty}}\|\nabla\bar{h}\|_{L^2}\\
&+\|h_1\|_{L^\infty}^3\|\nabla(h_1)^{-1}\|_{L^{6}}^2\|\bar{u}\|_{L^{6}}\|\nabla\bar{h}\|_{L^2}
+\|\nabla^2\bar{u}\|_{L^2}\|h_1\|_{L^\infty}\|\nabla\bar{h}\|_{L^2}\\
&+\|\nabla\bar{u}\|_{L^{3}}\|\nabla h_1\|_{L^{6}}\|\nabla\bar{h}\|_{L^2}\\
\leq& \epsilon\|\nabla^2 \bar{u}\|_{L^2}^2+CF(t)(\|\bar{h}\|_{H^1}^2+\|\nabla \bar{u}\|_{L^2}^2).
\end{split}
\end{equation}

Similar to \eqref{cx1}, multiplying \eqref{2.100} by $\bar{\varphi}$ and integrating by parts, one has
\begin{equation}
	\begin{split}
		\frac{d}{dt}\|\bar{\varphi}\|_{L^2}^2
\leq& \epsilon\|\nabla \bar{u}\|_{L^2}^2+CF(t)\Big(\|\bar{\varphi}\|_{L^2}^2+\|h_1\bar{u}\|_{L^2}^2\Big),
	\end{split}
\end{equation}
and
\begin{equation}
\begin{split}
\frac{d}{dt}\|\nabla\bar{\varphi}\|_{L^2}^2\leq \epsilon\|\nabla^2 \bar{u}\|_{L^2}^2+CF(t)(\|\bar{\varphi}\|_{H^1}^2+\|\nabla \bar{u}\|_{L^2}^2).
\end{split}
\end{equation}

Then, multiplying \eqref{2.99} by $\bar\psi$ and integrating by parts lead  to
\begin{equation}
\begin{split}\label{cx2}
	\frac{d}{dt}\|\bar\psi\|_{L^2}^2\leq&\|\nabla u_2\|_{L^\infty} \|\bar\psi\|_{L^2}^2+\|\bar{u}\|_{L^\infty}\|h_1\|_{L^\infty}\|\nabla^2(h_1)^{-1}\|_{L^{2}}\|\bar\psi\|_{L^2}\\
	&+\|\bar{u}\|_{L^{6}}\|h_1\|_{L^\infty}^2\|\nabla(h_1)^{-1}\|_{L^{6}}^2\|\bar\psi\|_{L^2}\\ &+\|\nabla\bar{u}\|_{L^{3}}\|h_1\|_{L^\infty}\|\nabla(h_1)^{-1}\|_{L^{6}}\|\bar\psi\|_{L^2}+\|\nabla^2\bar{u}\|_{L^2}\|\bar\psi\|_{L^2}\\
	\leq& \epsilon\|\nabla^2 \bar{u}\|_{L^2}^2+CF(t)\Big(\|\bar\psi\|_{L^2}^2+\|\nabla \bar{u}\|_{L^2}^2\Big).
\end{split}
\end{equation}

Finally, define
\begin{equation}\label{ggt}
\begin{split} G(t)\triangleq&\|h_1\bar{u}\|_{L^2}^2+\|\bar H\|_{L^2}^2+\eta\|\mathrm{curl}\bar H\|_{L^2}^2+\|\bar{h}\|_{H^1}^2+\|\bar{\varphi}\|_{H^1}^2+\|\bar\psi\|_{L^2}^2+\|\bar g\|_{L^2}^2\\
&+\nu\left(\mu\|\mathrm{curl}\bar u\|_{L^2}^2+(2\mu+\lambda)\|\mathrm{div}\bar u\|_{L^2}^2+ K(t)\right),
\end{split}
\end{equation}
by virtue of  \eqref{wsw2}, \eqref{wsa3}, \eqref{qaqa4u1}, \eqref{1wsw21}, and \eqref{cx1}-\eqref{cx2}, one derives from \eqref{ggt} by  choosing $\epsilon$ and $\nu$ small enough that
\begin{equation}
\frac{d}{dt}G(t)\leq CF(t)G(t).
\end{equation}
This combined with Gr\"onwall's inequality and \eqref{e116} provides $G(t)\equiv 0$. 
The proof of Theorem 1.1 is completed.

\section*{Conflict-of-interest statement}
All authors declare that they have no conflicts of interest.

\section*{Acknowledgments}
This work is partially supported by the National Natural Science Foundation of China
(No. 12371219), and the Double-Thousand Plan of Jiangxi Province (No. jxsq2023201115). 


\begin{thebibliography}{10}
\bibitem{Aramaki2014Lp}
J. Aramaki, $L^p$ theory for the div-curl system, Int. J. Math. Anal., 8 (2014), 259-271.
\bibitem{bresch2004some}
D. Bresch, B. Desjardins, Some diffusive capillary models of Korteweg type, C. R. Mec., 332 (2004), 881-886.
\bibitem{Cai2023}
G.C. Cai, J. Li,
Existence and exponential growth of global classical solutions to the compressible Navier-Stokes equations with slip boundary conditions in 3D bounded domains, Indiana Univ. Math. J., 72 (2023), 2491-2546.
\bibitem{Cai2112}
G.C. Cai, J. Li, B.Q. L\"u, Global classical solutions to the compressible Navier-Stokes equations with slip boundary conditions in 3D exterior domains, arXiv:2112.05586.
\bibitem{Chen20241} Y.Z. Chen, B. Huang, Y. Peng, X.D. Shi,
Global strong solutions to the compressible magnetohydrodynamic equations with slip boundary conditions in 3D bounded domains, J. Differ. Equ., 365 (2023), 274-325.
\bibitem{Chen2024}
Y.Z. Chen, B. Huang, X.D. Shi, Global Strong Solutions to the Compressible Magnetohydrodynamic Equations With Slip Boundary Conditions in a 3D Exterior Domain, Commun. Math. Sci., 22 (2024), 685-720.

\bibitem{3}
Y. Cho, H.J. Choe, H. Kim, Unique solvability of the initial boundary value problems for compressible viscous fluids, J. Math. Pures Appl., 83 (2004), 243-275.
\bibitem{4}
Y. Cho, H. Kim, On classical solutions of the compressible Navier-Stokes equations with nonnegative initial densities, Manuscr. Math., 120 (2006), 91-129.
\bibitem{5}
H.J. Choe, H. Kim, Strong solutions of the Navier-Stokes equations for isentropic compressible fluids, J. Differ. Equ., 190 (2003), 504-523.
\bibitem{Fan2009}
J.S. Fan, W.H. Yu, Strong solution to the compressible magnetohydrodynamic equations with vacuum, Nonlinear Anal: Real World Appl., 10 (2009), 392-409.
\bibitem{Feireisl2001}
E. Feireisl, A. Novotn\'y, H. Petzeltov\'a, On the existence of global defined weak solutions to the Navier-Stokes equations, J. Math. Fluid Mech., 3 (2001), 358-392.
\bibitem{hgyi} G.Y. Hong, X.F. Hou, H.Y. Peng, C.J. Zhu, Global existence for a class of large solutions to three-dimensional compressible magnetohydrodynamic equations with vacuum, SIAM J. Math. Anal., 49(4) (2017), 2409–2441.
\bibitem{hhpz}
G.Y. Hong, X.F. Hou, H.Y. Peng, C.J. Zhu, Global existence for a class of large solution to compressible Navier-Stokes equations with vacuum. Math. Ann., 388 (2024) no. 2, 2163-2194.
\bibitem{huwang2010} X.P. Hu, D.H. Wang, Global existence and large-time behavior of solutions to the three-dimensional equations of compressible magnetohydrodynamic flows, Arch. Ration. Mech. Anal., 197 (2010), 203–238.
\bibitem{HuangLiXin2012}
X.D. Huang, J. Li, Z.P. Xin, Global classical and weak solutions to the three-dimensional full compressible Navier-Stokes system with vacuum and large oscillations, Commun. Pure Appl. Math., 65 (2012), 549-585.
\bibitem{jiang1} S. Jiang, P. Zhang, On spherically symmetric solutions of the compressible isentropic Navier-Stokes
equations. Commun. Math. Phys., 215 (2001), 559-581.

\bibitem{jiang2} S. Jiang, P. Zhang, Axisymmetric solutions of the 3D Navier-Stokes equations for compressible
isentropic fluids. J. Math. Pures Appl., 82 (2003), 949-973.
\bibitem{10}
O. Ladyzenskaja, N. Ural'ceva, Linear and Quasilinear Equations of Parabolic Type, American Mathematical Society, Providence, 1968.
\bibitem{lxz2013siam} H.L. Li, X.Y. Xu, J.W. Zhang, Global classical solutions to 3D compressible magnetohydrodynamic equations with large oscillations and vacuum, SIAM J. Math. Anal., 45 (2013), 1356–1387.


\bibitem{13}
J. Li, Z.L. Liang, On local classical solutions to the Cauchy problem of the two-dimensional barotropic compressible Navier-Stokes equations with vacuum, J. Math. Pures Appl., 102 (2014), 640-671.
\bibitem{Li201977}
J. Li, Z.P. Xin, Global well-posedness and large time asymptotic behavior of classical solutions to the compressible Navier-Stokes equations with vacuum, Annals of PDE., 5 (2019), 7.
\bibitem{9}
J. Li, Z.P. Xin, Global existence of weak solutions to the barotropic compressible Navier-Stokes flows with degenerate viscosities, arXiv:1504.06826.

\bibitem{Lili2025}
J.X. Li, L.X. Li,
On local strong solutions to the Cauchy problem of 3D isentropic compressible Navier-Stokes equations with degenerate viscosities and far field vacuum, J. Differ. Equ., 435 (2025), 113338.

\bibitem{lly26ns}
J.X. Li, B.Q. L\"u, B. Yuan, Local well-posedness of strong solutions to the  compressible Navier-Stokes equations with degenerate viscosities and far field vacuum in 3D exterior domains,  arXiv:2602.04597.


\bibitem{6}
Y.C. Li, R.H. Pan, S.G. Zhu, On classical solutions to 2D shallow water equations with degenerate viscosities, J. Math. Fluid Mech., 19 (2017), 151-190.
\bibitem{7}
Y.C. Li, R.H. Pan, S.G. Zhu, On classical solutions for viscous polytropic fluids with degenerate viscosities and vacuum, Arch. Ration. Mech. Anal., 234 (2019), 1281-1334.
\bibitem{Lions1998-2}
P.L. Lions, Mathematical Topics in Fluid Mechanics. Vol. 2. Compressible Models, Oxford University Press, 1998.

\bibitem{liu2025}
H.R. Liu, T. Luo, H. Zhong, Strong solutions to the 3-D compressible MHD equations with density-dependent viscosities in exterior domains with far-field vacuum, arXiv:2504.04376.

\bibitem{liu2025ns}
H.R. Liu, H. Zhong, 
Well-posedness of the 3-D compressible Navier-Stokes equations with density-dependent viscosities in exterior domains with far-field vacuum,
arXiv:2512.07614.

\bibitem{lyz2013jde} S. Liu, H. Yu, J. Zhang, Global weak solutions of 3D compressible MHD with discontinuous initial data and vacuum.
J. Differ. Equ., 254(1), (2013) 229–255.


\bibitem{Lv2015}
B.Q. L\"u, B. Huang, On strong solutions to the Cauchy problem of the two dimensional compressible magnetohydrodynamic equations with vacuum, Nonlinearity., 28 (2015), 509-530.
\bibitem{lsxiumj} B.Q. L\:u, X.D. Shi, X.Y. Xu, Global existence and large-time asymptotic behavior of strong solutions to the com-pressible magnetohydrodynamic equations with vacuum, Indiana Univ. Math. J., 65(3) (2016), 925–975.
\bibitem{matsumura1980initial}
A. Matsumura, T. Nishida, The initial value problem for the equations of motion of viscous and heat-conductive gases, J. Math. Kyoto Univ., 20 (1980), 67-104.
\bibitem{Nash1962}
J. Nash, Le probl\'eme de Cauchy pour les \'equations diff\'erentielles d'un fluide g\'en\'eral, Bull. Soc. Math. Fr., 90 (1962), 487-497.
\bibitem{Navier1827}
Navier, C. L. M. H. Sur les lois de l'\'equilibre et du mouvement des corps\'elastiques, Mem. Acad. R. Sci. Inst. France., 6 (1827), 369.
\bibitem{Straskraba1993}
R. Salvi, I. Straskraba, Global existence for viscous compressible fluids and their behavior as $t \rightarrow \infty$, J. Fac. Sci., Univ. Tokyo, Sect. 1A, Math., 40 (1993), 17-51.
\bibitem{Serrin1959}
J. Serrin, On the uniqueness of compressible fluid motion, Arch. Ration. Mech. Anal., 3 (1959), 271-288.
\bibitem{Tani19774}
A. Tani, On the first initial-boundary value problem of compressible viscous fluid motion, Publ. Res. Inst. Math. Sci. Kyoto Univ., 13 (1977), 193-253.
\bibitem{Vasseur2016}
C. Vasseur, C. Yu, Existence of global weak solutions for 3d degenerate compressible Navier- Stokes equations, Invent. Math., (2016) 1-40.
\bibitem{wangxuejde} X. Wang, X.J. Xu, Global existence of strong solutions to the compressible magnetohydrodynamic equations with large initial data and vacuum in $\mathbb{R}^2$, J. Differ. Equ., 415 (2025), 722-763.
\bibitem{MR4340224}
Z.P. Xin, S.G. Zhu, Global well-posedness of regular solutions to the three-dimensional isentropic compressible Navier-Stokes equations with degenerate viscosities and vacuum, Adv. Math., 393 (2021), 108072.
\bibitem{zhuxin2021}
Z.P. Xin, S.G. Zhu, Well-posedness of three-dimensional isentropic compressible Navier-Stokes equations with degenerate viscosities and far field vacuum, J. Math. Pures Appl., 152 (2021), 94-144.




















\end{thebibliography}
\end{document}